\documentclass[reqno, 10.5pt]{amsart}

\usepackage{amsfonts,amsmath,amssymb,amsthm,amscd,enumerate,hyperref,comment}
\usepackage[dvipsnames]{xcolor}

\newtheorem{thm}{Theorem}[section]

\newtheorem{lem}{Lemma}[section]
\newtheorem{cor}{Corollary}[section]

\theoremstyle{remark}
\newtheorem{rmk}{Remark}[section]

\numberwithin{equation}{section}

\theoremstyle{definition}
\newtheorem{defn}{Definition}[section]

\newcommand{\C}{\ensuremath{\mathbb{C}}}

\newcommand{\R}{\ensuremath{\mathbb{R}}}

\newcommand{\rS}{\ensuremath{\mathbb{S}}}

\newcommand{\na}{\nabla}
\newcommand{\la}{\langle}

\newcommand{\ra}{\rangle}

\newcommand{\vphi}{\varphi}

\newcommand{\tr}{\operatorname{tr}}

\newcommand{\KN}{\mathbin{\bigcirc\mspace{-15mu}\wedge\mspace{3mu}}}

\DeclareMathOperator{\Ric}{Ric}

\DeclareMathOperator{\mC}{\mathcal{C}}

\begin{document}

\title[Curvature pinching of asymptotically conical Ricci expanders] {Curvature pinching of asymptotically conical gradient expanding Ricci solitons}
\author[Huai-Dong Cao and Junming Xie]{Huai-Dong Cao and Junming Xie}

\address{Department of Mathematics, Lehigh University, Bethlehem, PA 18015}
\email{huc2@lehigh.edu}

\address{Department of Mathematics, Rutgers University, Piscataway, NJ 08854}
\email{junming.xie@rutgers.edu}

\begin{abstract}

	In this paper, we investigate curvature pinching phenomena in complete non-compact {\it asymptotically conical} gradient expanding Ricci solitons and establish several Hamilton--Ivey type curvature pinching estimates. These results are parallel to those known for shrinking and steady Ricci solitons. In particular, we prove a three-dimensional Hamilton--Ivey type curvature pinching theorem: any three-dimensional non-compact gradient Ricci expander, which is asymptotic to a cone with positive scalar curvature, must have positive sectional curvature. Furthermore, we formulate a general method and apply it to obtain analogues of several additional known generalized Hamilton--Ivey type curvature pinching results for ancient solutions. Among these is a curvature pinching estimate for four-dimensional asymptotically conical Ricci expanders with uniformly positive isotropic curvature, analogous to a result for four-dimensional gradient steady solitons due to Brendle \cite{Brendle:14}.
	
\end{abstract}


\maketitle


\section{Introduction}

A complete Riemannian manifold $(M^n, g)$ is said to be a {\it gradient Ricci soliton} if there exists
a smooth function $f$ on $M^n$ such that the Ricci tensor $Rc$ of the metric $g$ satisfies the equation
\begin{equation} \label{eq:Riccisoliton}
	Rc + \nabla^2 f=\rho g
\end{equation}
for some constant $\rho \in{\mathbb R}$, where $\na^2 f$ denotes the Hessian of $f$. The Ricci soliton is said to be expanding, or steady, or shrinking if $\rho <0$, or  $\rho =0$, or $\rho>0$. 
The function $f$ is called a {\it potential function} of the gradient Ricci soliton.

Gradient Ricci solitons generate self-similar solutions to Hamilton's Ricci flow \cite{Ha:82} and play an important role in the study of the formation of singularities \cite{Ha:95F, Perelman:03}.  In particular, shrinking and steady solitons often arise as Type I and, respectively, Type II singularity models \cite{Ha:95F, Naber:10, Enders-Mueller-Topping:11, Cao:97} in the Ricci flow, while expanding solitons may arise as Type III singularity models \cite{Cao:97, Chen-Zhu:00} and over which the matrix Li-Yau-Hamilton (LYH) differential Harnack inequality \cite{Ha:93, Cao:92} becomes equality.  

The first examples of gradient expanding Ricci solitons are the one-parameter family of complete, rotationally symmetric, asymptotically conical gradient expanding Ricci solitons on $\R^n$ ($n \geq 3$) with positive (and negative) sectional curvature constructed by Bryant \cite{Bryant}, and the one-parameter family of complete, $U(n)$-invariant, asymptotically conical gradient expanding K\"ahler-Ricci solitons on $\C^n$ with similar curvature behavior constructed by the first author \cite{Cao:97}. The constructions in \cite{Cao:96, Cao:97} were later extended by Feldman-Ilmanen-Knopf \cite{FIK:03} to produce gradient  expanding K\"ahler-Ricci solitons on the complex line bundles $O(-k)$ ($k>n$) over the complex projective space ${\mathbb C}P^n$ ($n\geq 1$), and further generalized by Dancer-Wang  \cite{Dancer-Wang:11}. Additional constructions can be found in \cite{Lauret:01, Baird-Laurent:07, Dancer-Wang:09, Futaki-Wang:11, BDGW:15,  Wink:21, Wink:23, Nienhaus-Wink:24, Chi:24a}. 

Asymptotically conical gradient expanding Ricci solitons have received increasing attention in recent years. Chodosh \cite{Chodosh:14} showed that any gradient expanding Ricci soliton with positive sectional curvature that is asymptotic to a Euclidean cone must be rotationally symmetric.  A similar result for gradient K\"ahler-Ricci expanders was obtained by Chodosh-Fong \cite{Chodosh-Fong:16}. 
Schulze-Simon \cite{Schulze-Simon:13} constructed gradient expanding solitons emerging from the asymptotic cones at infinity of Ricci flow solutions on complete, non-compact, Riemannian manifolds with bounded,  nonnegative curvature operator  and positive asymptotic volume ratio. Deruelle \cite{Deruelle:16} proved that any Riemannian cone whose link is a differentiable sphere with curvature operator $Rm >1$ can be smoothed out by the Ricci flow into a gradient expanding Ricci soliton with nonnegative curvature operator. In the K\"ahler setting, Conlon, Deruelle, and Sun \cite{Conlon-Deruelle:20, Conlon-Deruelle-Sun:24} established the existence and uniqueness of asymptotically conical gradient expanding K\"ahler-Ricci solitons on smooth canonical models of K\"ahler cones. 

More recently, Chan-Lee-Peachey \cite{Chan-Lee-Peachey:24} showed that any metric cone at infinity of a non-collapsed weakly PIC1 manifold is resolved by a gradient expanding Ricci soliton. Bamler-Chen \cite{Bamler-Chen:23} developed a degree theory for 4-dimensional, asymptotically conical expanders, which implies the existence of gradient expanders asymptotic to any cone over $\rS^3$ with nonnegative scalar curvature. Additionally, Chan-Lee \cite{Chan-Lee:25} constructed various examples of asymptotically conical gradient expanders with positive curvature and exotic curvature decay. In particular, \cite{Gianniotis-Schulze:18, Angenent-Knopf:22, Bamler-Chen:23} have proposed that asymptotically conical gradient Ricci expanders may serve to continue Ricci flow past singular time and resolve conical singularities. For other related  developments, see \cite{Chow-Lu-Ni:06, CarNi:09, PRS:11, Bernstein-Mettler:15, Chen-Deruelle:15, CMM:16, Simon-Topping:21, Deruelle-Simon-Schulze:22, Deruelle-Schulze:23,  Leal-Vieira-Zhou:24} and the references therein. 

Curvature estimates for gradient expanding Ricci solitons with $Rc\geq 0$ or scalar curvature $R>0$ 
have also been established in \cite{Cao-Liu:22, Cao-Liu-Xie:23}, mirroring those for shrinking solitons \cite{Munteanu-Wang:15, Munteanu-Wang:17} or steady solitons \cite{Cao-Cui:20, Chan:19, Cao:22}. 
 
Despite the progress described above, a key feature well known for gradient shrinking and steady Ricci solitons has remained absent in the expanding case: Hamilton--Ivey type curvature pinching.  As is well-known, a hallmark of the three-dimensional Ricci flow is the {\it Hamilton--Ivey curvature pinching} \cite{Ha:95F, Ivey:93} (see also \cite[Theorem 2.4.1]{Cao-Zhu:06}), which asserts that when curvature blows up, the positive part blows up at a faster rate than the absolute value of the negative part.  In particular, 3-dimensional singularity models that are shrinking or steady gradient Ricci solitons, or more generally ancient solutions, must have nonnegative sectional curvature. This feature is remarkable:  it enables the application of the powerful Li--Yau--Hamilton differential Harnack inequality and the geometry of nonnegatively curved 3-manifolds to effectively analyze three-dimensional singularity models.  

The Hamilton--Ivey curvature pinching was later extended by B.-L. Chen \cite{ChenBL:09}  to arbitrary 3-dimensional ancient solutions\footnote{Ancient solutions exist for all {\it past} time, up to a final singular time; immortal solutions, starting at some initial time, exist for all {\it future} time.}, which include gradient shrinking and steady solitons as important special cases. It implies that any complete ancient solution in dimension three must have nonnegative sectional curvature. This has played a crucial role in the classifications of 3-dimensional gradient shrinking and steady Ricci solitons \cite{Cao-Chen-Zhu:08, Brendle:13, Lai:24, Lai:25, Lai:25b}, as well as 3D ancient solutions \cite{Brendle:20, Brendle-Dask-Sesum:21}.  

In higher dimensions, various forms of the generalized Hamilton--Ivey curvature pinching have been established for ancient solutions under suitable assumptions, see, e.g., \cite{Zhang:09, Bamler-CR-Wilking:19, Li-Ni:20, Cho-Li:23}. More recently, the authors proved certain curvature pinching properties for $4$-dimensional ancient solutions with positive isotropic curvature (PIC) or half-PIC \cite{Cao-Xie:23, Cao-Xie:25}, leading to partial classifications of $4$-dimensional shrinking and steady gradient Ricci solitons with weakly positive isotropic curvature (WPIC) or half-WPIC. 

However, analogous curvature pinching results for gradient expanding Ricci solitons have remained largely unknown. This is partly because the proofs of Hamilton--Ivey-type estimates for shrinking and steady solitons rely in an essential way on the {\it ancient} nature of these solutions. Since gradient expanders are not ancient solutions but special {\it immortal} solutions, those methods do not apply.  

Motivated in part by the second author's work \cite{Xie:25} on the convexity of mean convex, asymptotically conical self-expanders in mean curvature flow, and in part by the close resemblance between curvature estimates for Ricci expanders with nonnegative curvature \cite{Cao-Liu:22, Cao-Liu-Xie:23} and those for shrinking or steady Ricci solitons \cite{Munteanu-Wang:15, Munteanu-Wang:17, Cao-Cui:20, Chan:19, Cao:22} (see especially the comparison in dimension four given in \cite{Cao:22}), we began to investigate whether analogous Hamilton--Ivey-type curvature pinching properties might hold in the expanding case. By adapting an argument from the second author's recent work \cite{Xie:25} on asymptotically conical mean curvature expanders--inspired in turn by Spruck-Xiao \cite{Spruck-Xiao:20} and Xie-Yu \cite{Xie-Yu:23}--we have indeed established a number of generalized Hamilton--Ivey-type curvature pinching results for non-compact, asymptotically conical, gradient expanding Ricci solitons with positive scalar curvature.

\smallskip
Our first result is a Hamilton--Ivey type curvature pinching for $3$-dimensional complete non-compact, asymptotically conical, gradient Ricci expanders with positive scalar curvature. 

\begin{thm} \label{thm:ac-3D}
	Let $(M^3,g,f)$ be a $3$-dimensional non-compact, asymptotically conical\footnote{For all our results in this paper, except Corollary \ref{cor:ac_UPIC}, it suffices to assume $C^2$-asymptotics in the sense of Definition \ref{defn:conical}(a).} gradient expanding Ricci soliton. Suppose the asymptotic cone has positive scalar curvature. Then, $(M^3,g,f)$ must have positive sectional curvature.
\end{thm}

\begin{rmk}
	Theorem \ref{thm:ac-3D} may be viewed as an analogue of B.-L. Chen's result for three-dimensional ancient solutions to the Ricci flow \cite{ChenBL:09}. Moreover, after our paper was completed, we learned from P.-Y. Chan that a similar result to Theorem \ref{thm:ac-3D} was proved in \cite[Corollary 1.16]{Chan-Ma-Zhang:21} by a different method. 
\end{rmk}

For $n\geq 4$, we have the following Hamilton--Ivey type curvature pinching result for asymptotically conical, locally conformally flat gradient expanding solitons with positive scalar curvature.

\begin{thm} \label{thm:ac_Dflat}
	Let $(M^n,g,f)$, $n\geq 4$, be an $n$-dimensional non-compact, locally conformally flat, asymptotically conical gradient expanding Ricci soliton. Suppose the asymptotic cone has positive scalar curvature. Then, $(M^n,g,f)$ must have positive curvature operator.
\end{thm}

\begin{rmk}
    Theorem \ref{thm:ac_Dflat} is an analogue of Z.-H. Zhang's result for locally conformally flat gradient shrinking (and steady) solitons \cite{Zhang:09}. Moreover, by  \cite{Cao-Chen:13, Cao-Yu:21}, the locally conformally flat assumption can be replaced by the weaker assumption of the vanishing $D$-tensor introduced in \cite{Cao-Chen:12,Cao-Chen:13}.
\end{rmk}

As a consequence, by combining Theorem \ref{thm:ac_Dflat} with the work of Cao-Chen \cite{Cao-Chen:13} and Chen-Wang \cite{Chen-Wang:15}, we have the following application in dimension four.

\begin{cor} \label{cor:4D_halfLCF}
	Let $(M^4,g,f)$ be a $4$-dimensional non-compact, half-conformally flat, asymptotically conical gradient expanding Ricci soliton. Suppose the asymptotic cone has positive scalar curvature. Then, $(M^4,g,f)$ has positive curvature operator. 
\end{cor}

\begin{rmk}
	By \cite[Theorem 5.8]{Cao-etc:14},  the expanding solitons in Theorem \ref{thm:ac_Dflat} and Corollary \ref{cor:4D_halfLCF} are rotationally symmetric; see also Cao-Yu \cite[Corollary 3.4]{Cao-Yu:21}. 
\end{rmk}

Our next two results concern curvature pinching of 4-dimensional asymptotically conical gradient expanding Ricci solitons with either {\it positive isotropic curvature} (PIC) or {\it half-positive isotropic curvature} (half-PIC).

Recall that, for any oriented 4-manifold $(M^4,g)$, $2$-forms admit the decomposition $\wedge^2(M) = \wedge^{+}(M) \oplus \wedge^{-}(M)$, into self-dual and anti-self-dual $2$-forms. Accordingly, the Riemann curvature operator admits a block decomposition
\begin{equation*}
	Rm = 
	\begin{pmatrix}
		A & B \\
		B^t & C
	\end{pmatrix}
	=
	\begin{pmatrix}
		W^+ +\frac{R}{12}I & \mathring{Rc} \\
		\mathring{Rc} & W^- + \frac{R}{12}I
	\end{pmatrix},
\end{equation*}
where $W^{\pm}$ denote the self-dual and anti-self-dual part of the Weyl tensor, and $\mathring{Rc}$ the traceless Ricci part. It turns out that $(M^4,g)$ has PIC (half-PIC) if and only if $A$ and $C$ ($A$ or $C$) are 2-positive \cite{Ha:97}; see Section \ref{subsec:pic} for more details. 

\begin{thm} \label{thm:ac_halfPIC}
	Let $(M^4,g,f)$ be a $4$-dimensional non-compact asymptotically conical gradient expanding Ricci soliton with half-positive isotropic curvature (half-weakly PIC). 
	If the asymptotic cone has positive scalar curvature and satisfies either $A \geq 0$ or $C \geq 0$, then $(M^4,g,f)$ has $A > 0$ ($A\ge 0$) or $C > 0$ ($C\ge 0$).
\end{thm}

\begin{thm} \label{thm:ac_PIC}
	Let $(M^4,g,f)$ be a $4$-dimensional non-compact asymptotically conical gradient expanding Ricci soliton with positive isotropic curvature (weakly PIC). If the asymptotic cone has positive scalar curvature, then the Ricci curvature of $(M^4,g,f)$ is $2$-positive ($2$-nonnegative), and $|Rm| \leq  2R$. 
\end{thm}

\begin{rmk}
	Theorems \ref{thm:ac_halfPIC} and \ref{thm:ac_PIC} are analogues of our previous results for $4$-dimensional complete ancient solutions with half-PIC \cite[Proposition 3.1]{Cao-Xie:23} and PIC \cite[Theorem 1.1]{Cao-Xie:25}, respectively. Moreover, both results are valid under some slightly weaker assumptions; see Theorem \ref{thm:ac_A_2>0} and Theorem \ref{thm:ac_PIC_restate} in Section \ref{sec:ac-halfPIC}. 
\end{rmk}

By observing the common pattern in the proofs of Theorem \ref{thm:ac_halfPIC} and Theorem \ref{thm:ac_PIC}, we formulate a general method of proof (Lemma \ref{lem:general}) and apply it to obtain analogues of several additional known generalized Hamilton--Ivey type curvature pinching results for ancient solutions. This includes the following result for $4$-dimensional asymptotically conical gradient expanding Ricci solitons with uniformly positive isotropic curvature (UPIC), as well as several other results as stated in Theorem~\ref{thm:curv_pinching}.

\begin{thm} \label{thm:ac_UPIC}
	Let $(M^4,g,f)$ be a $4$-dimensional non-compact asymptotically conical gradient expanding Ricci soliton with uniformly positive isotropic curvature. If the asymptotic cone is a non-flat Euclidean cone, then the curvature operator of $(M^4,g,f)$ is positive. 
\end{thm}

\begin{rmk}
    Theorem \ref{thm:ac_UPIC} is an analogue of Brendle's curvature pinching result for  $4$-dimensional gradient steady Ricci solitons with UPIC \cite{Brendle:14}, as well as of Cho-Li's result for $4$-dimensional complete ancient solutions with UPIC \cite{Cho-Li:23}.
\end{rmk}

By combining Theorem \ref{thm:ac_UPIC}  with Chodosh's work \cite{Chodosh:14}, we obtain the following classification for 4D asymptotically conical expanding Ricci solitons with UPIC.

\begin{cor} \label{cor:ac_UPIC}
	Let $(M^4,g,f)$ be a $4$-dimensional non-compact gradient expanding Ricci soliton with uniformly positive isotropic curvature. If  $(M^4,g,f)$ is asymptotic\footnote{Here, the notion of asymptotically conical is $C^2$-asymptotic in the sense of Chodosh \cite[Definition 1.1]{Chodosh:14} which is slightly stronger than that in our Definition \ref{defn:conical}.} to a non-flat Euclidean cone, then it is rotationally symmetric.
\end{cor}

\begin{rmk}
	We note that any $4$-dimensional non-compact gradient expanding Ricci soliton with PIC that is asymptotic to a non-flat Euclidean cone automatically has UPIC. Therefore, one could replace the assumption of UPIC by PIC in both Theorem \ref{thm:ac_UPIC} and Corollary \ref{cor:ac_UPIC}. 
\end{rmk}

\noindent {\bf Organization of the Paper.} Section \ref{sec:pre} introduces the notation and basic concepts used throughout the paper, and collects several useful facts needed for the main arguments. Section \ref{sec:ac_Dflat} is devoted to the proofs of Theorem \ref{thm:ac-3D} and Theorem \ref{thm:ac_Dflat}. In Section \ref{sec:ac-halfPIC}, we present the proofs of Theorem \ref{thm:ac_halfPIC} and Theorem \ref{thm:ac_PIC}. In Section \ref{sec:general_lem}, we formulate a general lemma that can be applied especially to asymptotically conical gradient expanding Ricci solitons and use it to prove Theorem \ref{thm:ac_UPIC}  and Theorem \ref{thm:curv_pinching}. 
 Finally, the Appendix contains some elementary curvature properties of cones used in Section \ref{sec:ac_Dflat} and Section \ref{sec:ac-halfPIC}.

\medskip
\noindent {\bf Acknowledgements.} The first author was partially supported by a Simons Fellowship and a grant from the Simons Foundation. The second author would like to thank Prof. Xiaochun Rong for his continual support and encouragement. Both authors would like to thank Dr. P.-Y. Chan for very helpful information that allowed the removal of the positive scalar curvature assumption on $(M^n, f, g)$  in Theorem \ref{thm:ac-3D} and Theorem \ref{thm:ac_Dflat} in an earlier version of the paper. We also thank the anonymous referees for their helpful comments.

\bigskip
\section{Preliminaries} \label{sec:pre}
In this section, we fix notation and recall several basic facts and results that will be used throughout the paper. Throughout, we denote by 
$$Rm = \{R_{ijkl}\},\quad Rc = \{R_{ij}\},\quad R$$
the Riemann curvature tensor, the Ricci tensor, and the scalar curvature of the metric $g=\{g_{ij}\}$, respectively, either in local coordinates or with respect to a local orthonormal frame.

\subsection{Asymptotically conical expanding Ricci solitons} \label{subsec:ac_expander}
Recall that, in general, by an $n$-dimensional  \emph{cone} we mean an $n$-manifold
\begin{equation*}
	\mathcal{C} := [0,\infty) \times \Sigma^{n-1}
\end{equation*}
equipped with the Riemannian metric
\begin{equation*}
	g_c = dr^2 + r^2 \bar{g}_\Sigma,
\end{equation*}
where $(\Sigma^{n-1}, \bar{g}_\Sigma)$, called the link of  the cone $\mathcal{C}$, is a closed $(n-1)$-dimensional Riemannian manifold.  
As an example, the standard non-flat Euclidean cone with cone angle $\alpha \in [0,1)$ is given by the conical metric $g_\alpha$ on $\mathbb{R}^n\setminus \{0\}$, 
expressed in polar coordinates as
\begin{equation*} \label{eq:cone_alpha}
	g_\alpha := dr^2 + (1-\alpha) r^2 \bar{g}_{\rS^{n-1}(1)} .
\end{equation*}

In general, for any cone $\mathcal{C}$ over the link  $\Sigma$ and for $s \geq 0$, set
\begin{equation*}
	E_s = (s,\infty) \times \Sigma \subset\mathcal{C},
\end{equation*}
and define the dilation by $\tau > 0$ as the map
\begin{equation*}
	\rho_\tau : E_0 \to E_0, 
	\qquad 
	\rho_\tau(r,\sigma) = (\tau r, \sigma).
\end{equation*}

\begin{defn} \label{defn:conical}\
	\begin{itemize} 
		\item[(a)] A Riemannian manifold $(M^n, g)$ is said to be \emph{$C^k$-asymptotic to the cone $(E_0,g_c)$} if, for some $s > 0$, 
		there exists a diffeomorphism 
		\begin{equation*}
			\Phi : E_s \to M \setminus K,
		\end{equation*}
		for some compact subset $K \subset M$, 
		such that
		\begin{equation*}
			\tau^{-2} \rho_\tau^* \Phi^* g \longrightarrow g_c 
			\quad \text{as } \tau \to \infty
			\quad \text{in } C^k_{\mathrm{loc}}(E_0,g_c).
		\end{equation*}
		
		\item[(b)] We say that a Riemannian manifold $(M,g)$ is \emph{asymptotically conical} if there exists a cone $(E_0,g_c)$ such that $(M,g)$ is $C^k$-asymptotic to $(E_0,g_c)$ 
		for all integers $k\geq 0$.
	\end{itemize}
\end{defn}

\subsection{Curvature decomposition and isotropic curvature of four-manifolds} \label{subsec:pic}
In this subsection, we recall some facts about the curvature decomposition and isotropic curvature of $4$-manifolds. For more background, we refer the reader to Hamilton's paper \cite{Ha:97} and our previous work \cite{Cao-Xie:23,Cao-Xie:25}.

For any oriented Riemannian $4$-manifold $(M^4,g)$, the bundle of $2$-forms admits the decomposition
\begin{equation*}
	\wedge^2(M) = \wedge^{+}(M) \oplus \wedge^{-}(M),
\end{equation*}
where $\wedge^+(M)$ and $\wedge^-(M)$ denote the subbundles of self-dual and anti-self-dual $2$-forms, respectively. With respect to this splitting, the curvature operator has block form
\begin{equation} \label{eq:CODecomposition}
	{Rm} = 
	\begin{pmatrix}
		A & B \\
		B^{t} & C
	\end{pmatrix}
	=
	\begin{pmatrix}
		W^+ + \tfrac{R}{12}I & \mathring{Rc} \\
		\mathring{Rc} & W^- + \tfrac{R}{12}I
	\end{pmatrix},
\end{equation}
where $W^{\pm}$ are the self-dual and anti-self-dual Weyl tensors, $\mathring{Rc}$ denotes the traceless Ricci tensor\footnote{More precisely, $B:\wedge^-(M)\to \wedge^+(M)$ is given by $\mathring{Rc}\KN g$, the Kulkarni-Nomizu product of $\mathring{Rc}$ and $g$. In particular, $B\equiv 0$ whenever $(M^4,g)$ is Einstein.}, and $R$ is the scalar curvature.  

Let $A_1 \leq A_2 \leq A_3$ and $C_1 \leq C_2 \leq C_3$ denote the eigenvalues of $A$ and $C$, respectively.  It is also well-known that $\tr A = \tr C = {R}/{4}$.

\begin{defn}
	An $n$-dimensional Riemannian manifold $(M^n,g)$, $n \geq 4$, is said to have \emph{positive isotropic curvature} (PIC)  if
	\begin{equation*} \label{eq:PIC} 
		R_{1313}+R_{1414}+R_{2323}+R_{2424}-2R_{1234} > 0
	\end{equation*}
	for every orthonormal four-frame $\{e_1,e_2,e_3,e_4\}$. Similarly, it has \emph{nonnegative isotropic curvature}, or \emph{weakly PIC} (WPIC) if, for every such frame, 
	\begin{equation*} \label{eq:WPIC} 
		R_{1313}+R_{1414}+R_{2323}+R_{2424}-2R_{1234} \geq 0.
	\end{equation*}
\end{defn}

The notion of isotropic curvature was first introduced by Micallef-Moore \cite{Micallef-Moore:88}, in which they proved that any compact simply connected $n$-dimensional Riemannian manifold with PIC is homeomorphic to a round sphere. It also plays a key role in the convergence theory for the higher dimensional Ricci flow, especially in Brendle-Schoen's proof of the 1/4-pinching differentiable sphere theorem \cite{Brendle-Schoen:09}. 

It turns out that, in dimension four, these curvature conditions (and their natural extensions, half-PIC or half-WPIC) can be characterized in terms of the $3\times 3$ matrices $A$ and $C$ as follows:

\smallskip
\begin{itemize}
	\item {\em PIC} ({\em WPIC}) if and only if $A$ and $C$ are 2-positive (weakly 2-positive), i.e., $A_1 + A_2 > 0$ ($A_1 + A_2 \geq 0$) and  $C_1 + C_2 > 0$ ($C_1 + C_2 \geq 0$) on $M$  \cite{Ha:93};

	\smallskip
	\item {\em half-PIC} ({\em half-WPIC}) if and only if either $A$ or $C$ is 2-positive (weakly 2-positive), i.e., $A_1 + A_2 > 0$ ($A_1 + A_2 \geq 0$) or $C_1 + C_2 > 0$ ($C_1 + C_2 \geq 0$). 
\end{itemize}

\begin{defn}
	$(M^4, g)$ is said to be \emph{uniformly PIC} (UPIC) if $M^4$ has PIC and in addition satisfies the pointwise pinching condition  
	\begin{equation*} \label{eq:uniform-PIC}
		\max\{A_3,B_3,C_3\} \leq \Lambda \min\{A_1+A_2, C_1+C_2\}
	\end{equation*}
	on $M^4$ for some constant $\Lambda \geq 1$.
\end{defn}

\subsection{Basic differential equations satisfied by curvatures of  Ricci solitons } \label{subsec:id_solitons}
As a special case of curvature evolution equations under the Ricci flow \cite{Ha:86, Ha:97}, we have the following well-known curvature differential equations for any gradient Ricci soliton satisfying \eqref{eq:Riccisoliton}.

\begin{lem}[cf. Hamilton \cite{Ha:86}] \label{lem:4Dequations}
	Let $(M^4,g(t))$ be a $4$-dimensional complete gradient Ricci soliton satisfying Eq. \eqref{eq:Riccisoliton}. Then,
	\begin{gather*}
		\Delta_f R = 2\rho R - 2|Rc|^2, \\
		\Delta_f Rm = 2\rho Rm - 2(Rm^2 + Rm^{\sharp}), \\
		\Delta_f A = 2\rho A - 2(A^2 + 2A^{\sharp} + BB^t), \\
		\Delta_f B = 2\rho B - 2(AB + BC + 2B^{\sharp}), \\
		\Delta_f C = 2\rho C - 2(C^2 + 2C^{\sharp} + B^t B).
	\end{gather*}
	\noindent Here, for any $3\times 3$ matrix $D$, we denote by $D^2$ its square and by $D^{\sharp}$ the transpose of its adjoint. In addition, $\Delta_f := \Delta -\nabla f\cdot\nabla$ denotes the weighted Laplace operator.
\end{lem}

\begin{rmk}
	Except for the first identity in Lemma \ref{lem:4Dequations}, the factor $2$ differs from \cite{Ha:86} due to our normalization of the inner product on $\wedge^2(M)$ (see \eqref{eq:metric_2forms}). Moreover, the first two equations are valid in all dimensions.
\end{rmk}

\subsection{Calabi's barrier maximum principle} \label{subsec:calabi}
Finally, we shall need the following {\it barrier maximum principle} due to Calabi.

\begin{lem} [\cite{Calabi:58}] \label{lem:barrier-max}
	Let $\Omega \subset M$ be a bounded connected domain with smooth boundary, and 
	let $u \in C^0(\Omega)$. Let $L$ be a uniformly elliptic operator with continuous coefficients and vanishing constant term. If $L(u) \leq 0$ (resp. $L(u) \geq 0$) on $\Omega$ in the barrier sense, then
	\begin{equation*}
		\sup_{\Omega} u \geq \sup_{\partial \Omega} u 
		\qquad 
		\bigl(\text{resp. } \inf_{\Omega} u \leq \inf_{\partial \Omega} u \bigr).
	\end{equation*}
	Moreover, if $u$ attains an interior minimum (resp. maximum), then $u$ is a constant in $\Omega$.
\end{lem}

\section{Curvature pinching for asymptotically conical Ricci expanders with vanishing Weyl tensor} \label{sec:ac_Dflat}

This section is devoted to the proof of Theorem \ref{thm:ac-3D} and Theorem \ref {thm:ac_Dflat}. Note that both theorems follow from the following result.

\begin{thm} \label{thm:ac_Wflat}
	Let $(M^n,g,f)$, $n\geq 3$, be an $n$-dimensional non-compact asymptotically conical gradient expanding Ricci soliton with vanishing Weyl tensor $W\equiv 0$. Suppose the asymptotic cone has positive scalar curvature. Then, $(M^n,g,f)$ has positive curvature operator.	
\end{thm}

First of all, since the asymptotic cone of $(M^n,g,f)$ has positive scalar curvature, it follows from the strong maximum principle and \cite[Theorem 1.6]{Chan:23} that $(M^n,g,f)$ itself has positive scalar curvature $R>0$. This fact will be used in the proofs of Theorem \ref{thm:ac_Wflat} and Theorem \ref{thm:ac-3D}.

Next, we shall need the following differential inequality for the ratio between the smallest Ricci eigenvalue $\lambda_1$ and the scalar curvature $R$.

\begin{lem}[\cite{Eminenti-LNM:08,Petersen-W:10}] \label{lem:Delta_lambda1/R}
	Let $(M^n,g,f)$, $n\geq 3$, be an $n$-dimensional gradient Ricci soliton with vanishing Weyl tensor $W \equiv 0$ and positive scalar curvature $R>0$. Then, in the barrier sense, we have
	\begin{equation*}
		\Delta_F \left(\frac{\lambda_1}{R}\right) 
		\leq \frac{2h_1}{(n-1)(n-2)R^2} \leq 0,
		\qquad (F = f - 2\log R)
	\end{equation*}
	where
	\begin{equation*}
		\begin{split}
			h_1 &= (n-2)\lambda_1^2 (n\lambda_1 - R) + \bigl( (n-2)\lambda_1 - R \bigr)
			\Biggl( (n-1)\sum_{j=2}^n \lambda_j^2 - \Bigl(\sum_{j=2}^n \lambda_j\Bigr)^2 \Biggr).
		\end{split}
	\end{equation*}
\end{lem}

For the reader's convenience, we also include a proof here.

\begin{proof}
	Let $\lambda_1 \leq \lambda_2 \leq \cdots \leq \lambda_n$ be the Ricci eigenvalues. Fix any point $p_0 \in M$, and choose a unit vector field $e_1$ such that $Rc(e_1) = \lambda_1 e_1$ at $p_0$. Extend $e_1$ by parallel transport along geodesics emanating from $p_0$. Clearly, $\lambda_1 \leq Rc(e_1,e_1)$ with equality at $p_0$. Evaluating at $p_0$, we obtain, in the barrier sense (cf. Calabi \cite{Calabi:58}),
	\begin{equation} \label{eq:Delta_lambda1}
		\begin{split}
			\Delta_f \lambda_1 
			&\leq \Delta_f Rc(e_1,e_1) \\
			&= (\Delta_f Rc)(e_1,e_1) \\
			&= 2\rho \lambda_1 
			- \frac{2nR}{(n-1)(n-2)} \lambda_1
			+ \frac{4}{n-2} \lambda_1^2
			- \frac{2}{n-2}\left(|Rc|^2 - \frac{R^2}{n-1}\right),
		\end{split}
	\end{equation}
	where in the last equality, we have used \cite[(2.8)]{Eminenti-LNM:08} (see also \cite[Lemma 2.5]{Petersen-W:10}).
	
	On the other hand,  we have
	\begin{equation}\label{eq:quotient}
		\Delta_f\left(\frac{\lambda_1}{R}\right)
		= \frac{1}{R}\Delta_f \lambda_1
		- \frac{\lambda_1}{R^2}\Delta_f R
		- \frac{2}{R^2}\langle \nabla \lambda_1, \nabla R \rangle
		+ \frac{2\lambda_1}{R^3}|\nabla R|^2.
	\end{equation}
	Let $F = f - 2\log R$. Substituting \eqref{eq:Delta_lambda1} and the formula for $\Delta_f R$ in Lemma \ref{lem:4Dequations} into \eqref{eq:quotient}, we obtain
	\begin{equation*}\label{eq:ineq-simplified}
		\Delta_F \left(\frac{\lambda_1}{R}\right) \leq \frac{2h_1}{(n-1)(n-2)R^2},
	\end{equation*}
	where
	\begin{equation}\label{eq:h-def}
		h_1 = R^2(R-n\lambda_1)
		+(n-1)\Big(2\lambda_1^2 R + \big((n-2)\lambda_1 - R\big)|Rc|^2\Big).
	\end{equation}
	
	Now set 
	\begin{equation*} 
		S = \sum_{j=2}^n \lambda_j, 		\qquad  T = \sum_{j=2}^n \lambda_j^2	
	\end{equation*}
	so that 
	\begin{equation*}
		R = \lambda_1 + S, 
		\qquad |Rc|^2 = \lambda_1^2 + T. 
	\end{equation*}
	Substituting these expressions into \eqref{eq:h-def}, expanding and regrouping, we obtain
	\begin{equation*}
		\begin{split}
			h_1 &= S^2\big[S+(3-n)\lambda_1\big] + (2-n)\lambda_1^2 \big[S - (n-1)\lambda_1\big]\\
			&\quad + (n-1)T\big((n-3)\lambda_1 - S\big) \\
			&= (n-2)\lambda_1^2(n\lambda_1 - R) 
			+ \big[(n-2)\lambda_1 - R\big]\big[(n-1)T - S^2\big].
		\end{split}
	\end{equation*}
	Finally, it is clear from the above expression that $h_1 \leq 0$. 
\end{proof}

\medskip
Now, we divide the proof of Theorem \ref{thm:ac_Wflat} into three parts. 

\bigskip
\noindent {\bf Part I:} $(M^n, g, f)$ has nonnegative Ricci curvature $Rc\geq 0$.

\begin{proof}  [Proof of $Rc\geq 0$]
	As in \cite{Spruck-Xiao:20,Xie-Yu:23, Xie:25}, we argue by contradiction.  
	Suppose $Rc\ngeq 0$. Then the set
	$$ M^- := \{ p \in M : \lambda_1(p) < 0 \} $$
	is nonempty, hence
	\begin{equation*} \label{eq:epsilon_1<0}
		\epsilon_1 := \inf_{M}\left( \frac{\lambda_1}{R} \right) < 0.
	\end{equation*}
	
	\medskip
	\noindent\textbf{Case 1: Interior infimum.}  
	Suppose the negative infimum $\epsilon_1$ is attained at some point $p_0 \in M$. 
	By Lemma \ref{lem:Delta_lambda1/R} and Calabi's barrier strong maximum principle (Lemma \ref{lem:barrier-max}), 
	the ratio $\lambda_1/R$ must be constant on some bounded connected domain $p_0\in \Omega\subset M^-$. 
	In particular, $h_1 \equiv 0$ on $\Omega$ which forces
	\begin{equation} \label{eq:h=0}
		\lambda_1^2 (n\lambda_1 - R) \equiv 0, 
		\qquad
		\bigl((n-2)\lambda_1-R \bigr)\Biggl((n-1) \sum_{j=2}^n \lambda_j^2 - \Bigl(\sum_{j=2}^n \lambda_j\Bigr)^2\Biggr) \equiv 0.
	\end{equation}
	
	Since $\lambda_1<0$ at $p_0$ by assumption, from the first equation in \eqref{eq:h=0}, we must have $R=n\lambda_1<0$ at $p_0$, contradicting the fact that $R>0$ on $M$. Hence, \textbf{Case 1} is ruled out.
	
	\medskip
	\noindent\textbf{Case 2: Infimum at infinity.}  
	Assume instead that $\lambda_1/R$ attains its negative infimum at infinity. Following the argument of \cite{Xie:25}, there exists a sequence $\{p_i\}\subset M$ with $p_i \to \infty$ such that
	\begin{equation} \label{eq:lim_epsilon_1<0}
		\lim_{i\to\infty} \frac{\lambda_1}{R}(p_i) = \epsilon_1 < 0.
	\end{equation}
	Since the asymptotic cone $\mathcal{C}$ of the expanding soliton has positive scalar curvature, the ratio $\tilde{\lambda}_1/\tilde{R}$ is well defined on $\mathcal{C}$, where $\tilde{\lambda}_1$ and $\tilde{R}$ denote its smallest Ricci eigenvalue and scalar curvature, respectively. For the sequence $\{p_i\}\subset M$, one can find a corresponding sequence $\{\tilde{p}_i\} \subset \mathcal{C}$ such that
	\begin{equation*}
		\lim_{i\to\infty} \frac{\lambda_1}{R}(p_i)
		= \lim_{i\to\infty} \frac{\tilde{\lambda}_1}{\tilde{R}}(\tilde{p}_i).
	\end{equation*}
	
	On the other hand, since the expanding Ricci soliton satisfies $W \equiv 0$, its asymptotic cone $\mathcal{C}$ also has vanishing Weyl tensor. By Lemma \ref{lem:3Dcone} (for $n=3$) and Lemma \ref{lem:W=0cone} (for $n\ge 4$), $\mathcal{C}$ has nonnegative Ricci curvature, whose smallest eigenvalue $\tilde{\lambda}_1\equiv 0$ occurs in the radial direction. Hence
	\begin{equation*}
		\lim_{i\to\infty} \frac{\lambda_1}{R}(p_i) = 0,
	\end{equation*}

	\noindent contradicting \eqref{eq:lim_epsilon_1<0}. Thus, \textbf{Case 2} is also impossible.
	
	\smallskip
	Combining both cases, we conclude that $M^-=\emptyset$.  
	Therefore, $(M^n, g, f)$ has nonnegative Ricci curvature.
\end{proof}

\noindent {\bf Part II:} $(M^n, g, f)$ has nonnegative curvature operator $Rm\ge 0$.

\smallskip
First of all, when the Weyl tensor $W$ vanishes, we may express the smallest sectional curvature as (see, e.g., \cite[Proposition 3.1]{Zhang:09})
\begin{equation*}
	\mu := \min_{i,j} R_{ijij} = R_{1212} 
	= \frac{1}{n-2} \big[\lambda_1 + \lambda_2 - \frac{R}{n-1} \big],
\end{equation*}
where, as before, $\lambda_1 \leq \lambda_2 \leq \cdots \leq \lambda_n$ are the Ricci eigenvalues. 

\begin{rmk}
	We also note that, for any manifold with vanishing Weyl tensor, $\mu$ also coincides with the smallest eigenvalue of the curvature operator $Rm$ (see, e.g., \cite[Proposition 3.1(i)]{Zhang:09}). In particular, when $W=0$, nonnegative sectional curvature is equivalent to nonnegative curvature operator $Rm\geq 0$.
\end{rmk}

Next, we derive the following elliptic differential inequality for the smallest sectional curvature in the barrier sense.

\begin{lem} \label{lem:Delta_mu}
	Let $(M^n,g,f)$, $n \geq 3$, be an $n$-dimensional gradient Ricci soliton with vanishing Weyl tensor $W \equiv 0$ and positive scalar curvature $R>0$. Then, in the barrier sense, we have
	\begin{equation*}
		(n-2)\Delta_F \left(\frac{\mu}{R}\right) = \Delta_F \left(\frac{\lambda_1+\lambda_2}{R}\right) 
		\leq \frac{2E}{(n-1)(n-2)R^2},
		\qquad (F = f - 2\log R)
	\end{equation*}
	where
	\begin{equation*} \label{eq:E}
		\begin{split}
			E &=  -(n-2)\lambda_1^2 \sum_{j=3}^n(\lambda_j - \lambda_1) - (n-2)\lambda_2^2 \sum_{j=3}^n(\lambda_j - \lambda_2)\\
			&\quad - \bigl(R-(n-2)\lambda_1\bigr)
			\sum_{\substack{i>j \\ i,j \neq 1}} (\lambda_i - \lambda_j)^2 - \Biggl( \sum_{j=3}^n\bigl(\lambda_j - \lambda_2\bigr) \Biggr)
			\sum_{\substack{i>j \\ i,j \neq 2}} (\lambda_i - \lambda_j)^2 \\
			 &\quad - (\lambda_1+\lambda_2) \sum_{j=3}^{n} (\lambda_2 + \lambda_j - 2\lambda_1)(\lambda_j- \lambda_2) -(\lambda_1+\lambda_2) \sum_{\substack{i>j \geq 3}} (\lambda_i - \lambda_j)^2.
		\end{split}
	\end{equation*}
	Moreover, if the Ricci curvature is $2$-nonnegative then $E\leq 0$.
\end{lem}

\begin{proof}
	 First of all, since $\mu = \frac{1}{n-2} \big(\lambda_1 + \lambda_2 - \frac{R}{n-1} \big)$,  we have
	\begin{equation*}
	 	(n-2)\Delta_F \left(\frac{\mu}{R}\right) 
	 	= \Delta_F \left(\frac{\lambda}{R}\right),
	 \end{equation*}
	 where $F = f - 2\log R$ and $\lambda=\lambda_1+\lambda_2$. 

	Next, we proceed to compute $\Delta_F (\lambda/R)$. To begin with, for any fixed point $p_0 \in M$,  let $\{e_i\}$ be an orthonormal basis that diagonalizes the Ricci tensor so that $Rc(e_i) = \lambda_i e_i$ with $\lambda_1 \leq \cdots \leq \lambda_n$. We then extend $\{e_i\}$ to a local orthonormal frame in the neighborhood $U$ of $p_0$ by parallel transport along radial geodesics emanating from $p_0$. The resulting local orthonormal frame in $U$, denoted by $\{\tilde{e}_i\}$, satisfy
	\begin{equation} \label{eq:parallel}
		\nabla \tilde{e}_i(p_0) = \Delta \tilde{e}_i(p_0) = 0.
	\end{equation}
	
	Now we define
	\begin{equation*}
		\tilde{\lambda}(p) =
		\big( Rc(\tilde{e}_1,\tilde{e}_1) + Rc(\tilde{e}_2,\tilde{e}_2) \big)(p).
	\end{equation*}
	On the other hand, for any $p \in M$, it is clear that
	\begin{equation*}
		(\lambda_1 + \lambda_2)(p)
		= \inf \Big\{ Rc(\tau_1,\tau_1) + Rc(\tau_2,\tau_2) 
		\Big| \{\tau_i\} \text{ is an orthonormal basis of } T_pM \Big\}.
	\end{equation*}
	Consequently, at any $p\in U$, we have 
	\begin{equation*} \label{eq:u-tilde-bound}
		\tilde{\lambda}(p) \geq \lambda(p) = \lambda_1(p) + \lambda_2(p),
	\end{equation*}
	with equality at $p_0$, i.e., $\tilde{\lambda}$ is a barrier function for $\lambda$.
	
	Evaluating at $p_0$ in the barrier sense, using \eqref{eq:parallel} and \cite[(2.8)]{Eminenti-LNM:08} (see also \cite[Lemma 2.5]{Petersen-W:10}), we obtain
	\begin{equation*} \label{eq:Delta_lambda}
		\begin{split}
			\Delta_f \lambda & = \Delta_f (\lambda_1 +\lambda_2) \\
			&\leq \Delta_f \big( Rc(\tilde{e}_1,\tilde{e}_1) + Rc(\tilde{e}_2,\tilde{e}_2)\big) \\
			&= (\Delta_f Rc)(\tilde{e}_1,\tilde{e}_1) + (\Delta_f Rc)(\tilde{e}_2,\tilde{e}_2) \\
			&= 2\rho \lambda_1 
			- \frac{2nR}{(n-1)(n-2)} \lambda_1
			+ \frac{4}{n-2} \lambda_1^2
			- \frac{2}{n-2}\left( |Rc|^2 - \frac{R^2}{n-1} \right) \\
			&\quad + 2\rho \lambda_2 
			- \frac{2nR}{(n-1)(n-2)} \lambda_2
			+ \frac{4}{n-2} \lambda_2^2
			- \frac{2}{n-2}\left( |Rc|^2 - \frac{R^2}{n-1} \right).
		\end{split}
	\end{equation*}
	
	Let $F = f - 2\log R$. By direct computations as in the proof of Lemma \ref{lem:Delta_lambda1/R} and using the above equation, we obtain
	\begin{equation*}
		\begin{split}
			\Delta_F \left(\frac{\lambda}{R}\right)
			&= \frac{1}{R}\Delta_f \lambda - \frac{\lambda}{R^2}\Delta_f R \\
            &= \frac{1}{R}\Delta_f (\lambda_1+\lambda_2) 
            - \frac{\lambda_1+\lambda_2}{R^2}\Delta_f R \\
			&\leq \frac{2E}{(n-1)(n-2)R^2},
		\end{split}
	\end{equation*}
	where $E=h_1+h_2$ and, for $i=1, 2$,  
	\begin{equation*}
    	h_i = (n-2)\lambda_i^2 (n\lambda_i - R) + \bigl((n-2)\lambda_i - R \bigr)
 		\Biggl( (n-1)\sum_{j\neq i}^n \lambda_j^2 - \Bigl(\sum_{j\neq i}^n \lambda_j\Bigr)^2 \Biggr).
	\end{equation*}
	
	It remains to express $E$ in the form as given in Lemma \ref{lem:Delta_mu}. For this purpose, we rewrite 
	\[ n\lambda_1 - R= -(\lambda_2-\lambda_1) - \sum_{j=3}^n (\lambda_j-\lambda_1),\]
	\[ n\lambda_2 - R= (\lambda_2-\lambda_1) - \sum_{j=3}^n (\lambda_j-\lambda_2),\]
	and observe that 
	\begin{equation*}
		(n-1)\sum_{j\neq 1}^n \lambda_j^2
		- \Bigl(\sum_{j\neq 1}^n \lambda_j\Bigr)^2
		= \sum_{\substack{i>j \\ i,j \neq 1}} (\lambda_i - \lambda_j)^2,
	\end{equation*}
	\begin{equation*}
		(n-1)\sum_{j\neq 2}^n \lambda_j^2
		- \Bigl(\sum_{j\neq 2}^n \lambda_j\Bigr)^2
		= \sum_{\substack{i>j \\ i,j \neq 2}} (\lambda_i - \lambda_j)^2.
	\end{equation*}

	By substituting the above identities into $E=h_1+h_2$, we get
	\begin{equation} \label{eq:Esimplify}
		\begin{split}
			E & =  (n-2) (\lambda_2-\lambda_1)^2 (\lambda_1+\lambda_2) 
       - (n-2)\lambda_1^2 \sum_{j=3}^n (\lambda_j - \lambda_1) \\
			&\quad - (n-2)\lambda_2^2 \sum_{j=3}^n (\lambda_j - \lambda_2) + \bigl((n-2)\lambda_1 - R \bigr)
			\sum_{\substack{i>j \\ i,j \neq 1}} (\lambda_i - \lambda_j)^2 \\
			&\quad + \bigl((n-2)\lambda_2 - R \bigr)
			\sum_{\substack{i>j \\ i,j \neq 2}} (\lambda_i - \lambda_j)^2.
	\end{split}
	\end{equation}

	One can notice that every term, except possibly the first one and the last one, in the above expression \eqref{eq:Esimplify} of $E$ is nonpositive. On the other hand, the last term can be rewritten as 
	\begin{equation}\label{eq:Elastterm}
		\begin{split}
			\bigl((n-2)\lambda_2 - R \bigr) \sum_{\substack{i>j \\ i,j \neq 2}} (\lambda_i - \lambda_j)^2 
			& = -\Bigg(\lambda_1+\lambda_2 +\sum_{j=3}^n (\lambda_j - \lambda_2)\Bigg) \sum_{\substack{i>j \\ i,j \neq 2}} (\lambda_i - \lambda_j)^2 \\
			& = - (\lambda_1+\lambda_2)\sum_{\substack{i>j \\ i,j \neq 2}} (\lambda_i - \lambda_j)^2 \\
			&\quad - \Biggl( \sum_{k=3}^n (\lambda_k- \lambda_2) \Biggr)
			\sum_{\substack{i>j \\ i,j \neq 2}} (\lambda_i - \lambda_j)^2.
		\end{split}
	\end{equation}
	Combining \eqref{eq:Esimplify} and  \eqref{eq:Elastterm}, we arrive at the desired expression of $E$. Also, it is clear $E\leq 0$ when $\lambda_1+\lambda_2\geq 0$, i.e., the Ricci curvature is 2-nonnegative. This finishes the proof of the lemma. 
\end{proof}

Now, we are ready to complete the proof that  $(M^n,g,f)$ has $Rm\ge 0$.

\begin{proof} [Proof of $Rm\ge 0$]
	First, we claim that $(M^n,g,f)$ must have nonnegative sectional curvature. Suppose not. Then, as before, we argue by contradiction. Assume the claim fails. Then the set
	\begin{equation*}
		M^- := \{ p \in M : \mu(p) < 0 \}
	\end{equation*}
	is nonempty, where
	\begin{equation*}
	\mu = \min_{i,j} R_{ijij} = R_{1212}
	= \frac{1}{n-2} \big[\lambda_1 + \lambda_2 - \frac{R}{n-1} \big]
	\end{equation*}
	is the smallest sectional curvature. Moreover,
	\begin{equation} \label{eq:epsilon_0}
		\epsilon_0 := \inf_{M}\left( \frac{\mu}{R} \right) < 0 .
	\end{equation}
	
	\smallskip
	\noindent\textbf{Case 1: Interior infimum.}
	Suppose the negative infimum $\epsilon_0$ is attained at some point $p_0 \in M$. By {\bf Part I} of the proof above, we know that the Ricci curvature is nonnegative. Then, by Lemma \ref{lem:Delta_mu}, we have
	\begin{equation*}
		(n-2)\Delta_F \left(\frac{\mu}{R}\right) 
		= \frac{2E}{(n-1)(n-2)R^2} \leq 0.
		\qquad (F = f - 2\log R)
	\end{equation*}
	By Calabi's barrier maximum principle (Lemma \ref{lem:barrier-max}), the ratio $\mu/R$ must be constant on some bounded connected domain $p_0\in \Omega\subset M^-$. Thus, we have $E \equiv 0$ on $\Omega$ which, by the expression of $E$ in Lemma \ref{lem:Delta_mu}, forces
	\begin{equation}\label{eq:E=0-terms}
		\begin{split}
			(n-2)\lambda_2^2 \sum_{j=3}^n (\lambda_2 - \lambda_j) &\equiv 0, \\[0.3em]
			(n-2)\lambda_1^2 \sum_{j=3}^n (\lambda_1 - \lambda_j) &\equiv 0, \\[0.3em]
			\bigl((n-2)\lambda_1 - R \bigr)
			\sum_{\substack{i>j \\ i,j \neq 1}} (\lambda_i - \lambda_j)^2 &\equiv 0.
		\end{split}
	\end{equation}
	
	From the first identity in \eqref{eq:E=0-terms}, either $\lambda_2 = 0$ or $\lambda_2 = \lambda_j$ for all $j\ge 3$.
	
	\medskip
	$\bullet$ {Subcase 1a: $\lambda_2=0$ at $p_0$.}
	By {\bf Part I}, we have $\Ric \ge 0$, which implies $\lambda_1(p_0)=0$.
	Then, the third relation in \eqref{eq:E=0-terms} yields $\lambda_i=\lambda_j$ at $p_0$ for all $i,j\ge 2$, hence $\lambda_i(p_0) = 0$ for all $i$, contradicting the fact that the scalar curvature $R$ is positive.
	
	\medskip
	$\bullet$ {Subcase 1b: $\lambda_2 = \lambda_j$, at $p_0$, for all $j \ge 3$.}
	The second relation in \eqref{eq:E=0-terms} implies,  at $p_0$,  either $\lambda_1=0$ or
	$\lambda_1=\lambda_j$  for $j\ge 3$.
	If $\lambda_1=\lambda_j$ at $p_0$ for all $j\ge 3$, then
	\begin{equation*}
	\lambda_1(p_0) = \lambda_2(p_0) = \lambda_j(p_0) \qquad (j\ge 3).
	\end{equation*}
	So at $p_0$, we have $0>\lambda_1(p_0) = \cdots =\lambda_n(p_0)$ contradicting $R>0$.
	Suppose instead that $\lambda_1(p_0)=0$ and $\lambda_2(p_0)=\lambda_j(p_0)$ for all $j\ge 3$. Then, at $p_0$, we have
	\begin{equation*}
	\mu= \frac{1} {n-2} \big[\lambda_1 + \lambda_2 - \frac{R}{(n-1)}\big] =0,
	\end{equation*}
	which is a contradiction to \eqref{eq:epsilon_0}. Therefore, \textbf{Case 1} is impossible.
	
	\medskip
	\noindent\textbf{Case 2: Infimum at infinity.} 
	Suppose $\mu/R$ attains its negative infimum at infinity. Then, by Lemma \ref{lem:3Dcone} and Lemma \ref{lem:W=0cone} and essentially the same argument as in the proof of $Rc\geq 0$ in {\bf Part I}, we can rule out {\bf Case 2}. We omit the details.  
	
	\medskip
	Hence, we conclude that $M^-=\emptyset$. Therefore, $M$ has nonnegative sectional curvature.   
	As nonnegative sectional curvature is equivalent to nonnegative curvature operator for manifolds with vanishing Weyl tensor $W=0$, it follows that the expanding soliton $(M, g, f)$ has nonnegative curvature operator $Rm\geq 0$. In fact, $\mu$ coincides with the smallest eigenvalue of the curvature operator $Rm$ (see, e.g., \cite[Proposition 3.1(i)]{Zhang:09}).
\end{proof}

\medskip
\noindent {\bf Part III:} $(M^n, g, f)$ has positive curvature operator $Rm > 0$.		
	
\begin{proof} [Proof of $Rm > 0$]
	To begin with, following the proof of \cite[Theorem 1.3]{Cao-Xie:25}, we observe that 
	the null space of the curvature operator, $\ker(Rm)$, is invariant under parallel translation. 
	Indeed, consider the canonical immortal solution 
	\[
	g(t) = (1+t)\Phi(t)^{\ast}(g), \quad -1<t< \infty,
	\]
	to the Ricci flow induced by the expanding soliton $(M^n,g,f)$, with $g(0)=g$ for $t \in [0,1]$.  
	By the evolution equation of the curvature operator (see \cite{Ha:86})
	\[
	\partial_t Rm(t) = \Delta Rm(t) + 2\big(Rm^2(t)+Rm^{\sharp}(t)\big)
	\]
	and the fact that $(M, g)$ has nonnegative curvature operator $Rm\geq 0$, it follows that the quadratic term
	$Rm(t)^2+Rm(t)^{\sharp} \geq 0$ for all $t \in [0,1]$.  
	Thus, by Hamilton's strong maximum principle (see, e.g., \cite[Theorem 2.2.1]{Cao-Zhu:06}), 
	there exists an interval $0<t<\delta$ over which the rank of $Rm(t)$ is constant, 
	and $\ker(Rm(t))$ is invariant under parallel translation.  
	Since $g(t)=(1+t)\Phi(t)^{\ast}(g)$ with $g(0)=g$, we conclude that the rank of $Rm = Rm(0)$ is locally constant and that $\ker(Rm) = \ker(Rm(0))$ is invariant under parallel translation.
	
	\medskip
	\noindent{\bf Claim.} The expanding Ricci soliton $(M^n,g,f)$ has positive curvature operator.
	
	\smallskip
	\noindent{\it Proof of Claim.}
	As in \cite{Xie:25}, we argue by contradiction. Suppose that $M$ does not have positive curvature operator. Since $M$ has nonnegative curvature operator, there exists a point $p_0\in M$ such that the smallest eigenvalue of the curvature operator $\mu(p_0)=0$. 
	Because the rank of the curvature operator $Rm$ is locally constant, it follows that $\mu\equiv 0$ on a neighborhood $U$ of $p_0$. Moreover, as $\ker(Rm)$ is invariant under parallel translation, there exists an orthogonal decomposition of the tangent bundle $TU=V_1\oplus V_2$, where $V_1=\ker(Rm)$ with $\dim(V_1)\geq 1$, and both $V_1$ and $V_2$ are invariant under parallel translation. By \cite[Lemma 9.1]{Ha:86}, $M$ locally splits off at least one Euclidean factor. Since expanding Ricci solitons are real analytic, completeness implies that this local splitting extends globally.
	However, this contradicts the assumption that $M$ is asymptotically conical: 
	any Euclidean factor of $M$ would force the asymptotic cone $\mathcal{C}$ to split off a Euclidean factor, thereby violating the regularity assumption that $\mathcal{C}$ has only an isolated singularity at the tip, unless $\mathcal{C}$ is a Euclidean space with flat metric. In this latter case, a flat Euclidean space is not a cone with positive scalar curvature, which is a contradiction. Consequently, $M$ must have positive curvature operator.

	\smallskip	
	This completes the proof of the {\bf Claim}, and the proof of Theorem \ref{thm:ac_Wflat}.
\end{proof}

\begin{proof} [Proof of Corollary \ref{cor:4D_halfLCF}] 
	Let $(M^4, g, f)$ be a non-compact, half-conformally flat (i.e., either $W^+=0$ or $W^-=0$), asymptotically conical Ricci expander with positive scalar curvature. By essentially the same arguments as in Chen-Wang \cite{Chen-Wang:15} for gradient shrinking and steady Ricci solitons, one can show that $(M^4, g, f)$ has vanishing $D$-tensor as introduced in \cite{Cao-Chen:12, Cao-Chen:13}. Thus, by  \cite[Proposition 3.2]{Cao-Chen:13},  $(M^4, g, f)$ is locally conformally flat. Therefore, it follows from Theorem \ref{thm:ac_Wflat} that $(M^4, g, f)$ has positive curvature operator $Rm>0$. 
\end{proof}

Finally, we conclude this section by providing a more direct proof of Theorem~\ref{thm:ac-3D}.

\begin{proof} [Proof of Theorem \ref{thm:ac-3D}] 
	First, note that we can diagonalize the curvature operator $Rm$ with eigenvalues $m_1\leq m_2\leq m_3$, which are principal sectional curvatures of $(M^3, g, f)$, with respect to a suitable orthonormal frame $\{e_1, e_2, e_3\}$ at each point $p\in M^3$. Then, the Ricci tensor $Rc$ is diagonalized with eigenvalues 
	$$(m_1+m_2)\leq (m_1+m_3)\leq (m_2+m_3),$$  
	and the scalar curvature is given by $R=2(m_1+m_2+m_3)$. Moreover, for $n=3$, the differential equation 
	$$\Delta_f Rm = 2\rho Rm - 2(Rm^2 + Rm^{\sharp}),$$
	implies (\cite{Ha:86})
	\begin{equation*} \label{eq:3d_eqs}
		\Delta_fm_1\leq 2\rho m_1 - 2(m_1^2+m_2m_3).
	\end{equation*}
	Then, by direct computations, we have
	\begin{equation} \label{eq:m_1/R}
		\begin{split}
			\Delta_F \left( \frac{m_1}{R} \right) &= \frac{1}{R}\Delta_fm_1 - \frac{m_1}{R^2} \Delta_f R \quad\quad (F = f - 2\log R)\\ 
			&\leq \frac{2}{R^2} \left[m_1|Rc|^2 - R(m_1^2+m_2m_3)\right] \\
			&= \frac{4}{R^2} \left[(m_1-m_3)m_2^2 + (m_1-m_2)m_3^2\right] \\
            & \leq 0. 
		\end{split}
	\end{equation} 
	
	Next, we show $m_1\geq 0$ (hence $(M^3,g,f)$ has nonnegative curvature operator) by contradiction. Suppose not. Then, the set
	\begin{equation*}
	M^- := \{ p \in M : m_1(p) < 0 \}
	\end{equation*}
	is nonempty, and
	\begin{equation*} \label{eq:epsilon}
		\epsilon := \inf_{M}\left( \frac{m_1}{R} \right) < 0 .
	\end{equation*}
	
	\noindent{\bf Case 1: Interior infimum.}  
	Suppose the negative infimum $\epsilon$ is attained at some $p_0 \in M$ such that $m_1(p_0)<0$. By \eqref{eq:m_1/R}, we have
	\begin{equation*}
		\begin{split}
			\Delta_F \left( \frac{m_1}{R} \right) \leq \frac{4}{R^2} \left[(m_1-m_3)m_2^2 + (m_1-m_2)m_3^2\right]  \leq 0.
		\end{split}
	\end{equation*}
	Then, Calabi's barrier strong maximum principle (Lemma \ref{lem:barrier-max}) implies the ratio $m_1/R$ must be constant on a neighborhood of $p_0$. In particular, in that neighborhood of $p_0$, we have
	\begin{equation} \label{eq:m_eqs}
		(m_1-m_3) m_2^2 \equiv 0,\quad (m_1-m_2)m_3^2 \equiv 0.
	\end{equation}
	
	Since $m_1(p_0) < 0$ and the scalar curvature $R>0$, the first equation in \eqref{eq:m_eqs} forces $m_2(p_0) = 0$. However, the second equation in \eqref{eq:m_eqs} then implies $m_3(p_0) = 0$, which contradicts the assumption of positive scalar curvature. Thus, \textbf{Case 1} is ruled out.
	
	\smallskip
	\noindent{\bf Case 2: Infimum at infinity.}  
	Assume instead that $m_1/R$ attains its negative infimum at infinity. By Lemma \ref{lem:3Dcone}, any 3-dimensional cone with positive scalar curvature automatically has positive curvature operator/sectional curvature, and the smallest eigenvalue of its curvature operator is zero. Then, as in {\bf Part I} of the proof of Theorem \ref{thm:ac_Wflat}, we get a contradiction. Thus, \textbf{Case 2} is impossible.  
	
	\smallskip
	Combining both cases, we conclude that $M^-=\emptyset$, so $M$ has nonnegative curvature operator or, equivalently, nonnegative sectional curvature (for $n=3$).  
	
	\smallskip
	Finally, by the same argument as in the proof of $Rm>0$ (i.e., {\bf Part III})  for Theorem \ref{thm:ac_Wflat}, we conclude that $M^3$ has positive curvature operator. This completes the proof of Theorem \ref{thm:ac-3D}.
\end{proof}

\smallskip
\section{4D asymptotically conical Ricci expanders with (half) PIC} \label{sec:ac-halfPIC}

In this section, we study curvature pinching in  four-dimensional asymptotically conical gradient Ricci expanders with (half) PIC. In particular, we prove Theorem \ref{thm:ac_halfPIC} and Theorem \ref{thm:ac_PIC}.

We first investigate the positivity of the self-dual curvature operator $A$ (respectively, the anti-self-dual curvature operator $C$), as defined in the curvature decomposition \eqref{eq:CODecomposition},  for asymptotically conical gradient expanding Ricci solitons with $A_2 \geq 0$ (respectively, \ $C_2 \geq 0$).  We assume in addition that the asymptotic cone has nonnegative self-dual (or anti-self-dual) curvature operator and positive scalar curvature. Under these conditions, we prove the following stronger result, from which Theorem \ref{thm:ac_halfPIC} follows. 

\begin{thm} \label{thm:ac_A_2>0}
	Let $(M^4,g,f)$ be a $4$-dimensional non-compact asymptotically conical gradient expanding Ricci soliton. Suppose that the asymptotic cone has positive scalar curvature and satisfies either $A \geq 0$ or $C \geq 0$. Then:
	\begin{itemize}
		\item[(a)] If $(M^4,g,f)$ satisfies $A_2 \geq 0$ or $C_2 \geq 0$, then $A \geq 0$ or $C \geq 0$ on $M^4$.
		
		\item[(b)] If $(M^4,g,f)$ satisfies $A_2 > 0$ or $C_2 > 0$, then $A > 0$ or $C > 0$ on $M^4$.
	\end{itemize}
\end{thm}

\begin{proof} 
	Again, as first mentioned in Section \ref{sec:ac_Dflat}, the assumption of the asymptotic cone having positive scalar curvature implies that $(M^n,g,f)$ itself has positive scalar curvature $R>0$.

	(a) Without loss of generality, we may assume $A_2 \geq 0$. Again, as in the proof of Theorem \ref{thm:ac_Wflat}, we argue by contradiction. Suppose instead that the claim fails. Then the set
	\begin{equation*}
	M^- := \{ p \in M : A_1(p) < 0 \}
	\end{equation*}
	is nonempty, and we have
	\begin{equation*}
	\epsilon^{\prime} := \inf_{M}\left( \frac{A_1}{R} \right) < 0.
	\end{equation*}

	\smallskip
	\noindent\textbf{Case 1: Interior infimum.}  
	Suppose the negative infimum $\epsilon^{\prime}$ is attained at some point $p_0 \in M$. Then, there exists a neighborhood $\Omega \ni p_0$, such that $A_1(p)<0$ for all $p \in \Omega$. Thus, by Lemma \ref{lem:4Dequations} and direct computation, for $F = f - 2\log R$, we have
	\begin{equation*}
		\Delta_F \left( \frac{A_1}{R} \right) 
		\leq \frac{2}{R^2} \left[ A_1|Rc|^2 - R(A_1^2 + B_1^2 + 2A_3A_2) \right] \leq 0,
	\end{equation*}
	in the barrier sense on $\Omega$.

	Now, by Calabi's barrier strong maximum principle (Lemma \ref{lem:barrier-max}), the ratio $A_1/R$ must be constant on $\Omega$. In particular, we have
	\begin{equation*}
		A_1|Rc|^2 - R(A_1^2 + B_1^2 + 2A_3A_2) \equiv 0,
	\end{equation*}
	which implies $A_1 |Rc|^2 = R(A_1^2 + B_1^2 + 2A_3A_2) \equiv 0$ on $\Omega$.  

	Since $A_1/R$ attains its negative infimum at $p_0$, $A_1(p_0)\neq 0$ so we must have $|Rc|^2=0$. then, the scalar curvature is zero. This contradicts the fact that the scalar curvature is positive. Hence, {\bf Case 1} is ruled out.

	\smallskip
	\noindent\textbf{Case 2: Infimum at infinity.}  
	If $A_1/R$ attains its negative infimum at infinity, then by the assumptions on the asymptotic cone and the same argument in the proof of Theorem  \ref{thm:ac_Wflat}, we can rule out {\bf Case 2}.

	\smallskip
	Combining both cases, we conclude that $M^-=\emptyset$. Therefore, $M$ satisfies $A \geq 0$, completing the proof of part (a).

	\smallskip
	(b) Without loss of generality, we assume $A_2 > 0$. By part (a), this implies $A_1 \geq 0$. Following \cite[Proposition 3.1(b)]{Cao-Xie:23}, we prove that $A_1 > 0$ by contradiction. Suppose $A_1(p_0) = 0$ at some point $p_0 \in M^4$. Then $A_1$ attains its minimum at $p_0$. Let $\eta \in \wedge_{p_0}^{+}(M)$ be a null eigenvector of $A$ such that $A(\eta,\eta) = A_1(p_0) = 0$ at $p_0$. Extend $\eta$ to a local section (still denoted by $\eta$) by parallel transport along geodesics emanating from $p_0$.  

	At $p_0$, in the barrier sense, Lemma \ref{lem:4Dequations} yields
	\begin{equation*}
		\begin{split}
			0 &\leq \Delta_f A_1 \\
			&\leq \Delta_f A(\eta, \eta) \\
			&= (\Delta_f A)(\eta, \eta) \\
			&\leq 2\bigl(\rho A_1 - A_1^2 - 2A_2 A_3 - B_1^2\bigr) \\
			&< 0,
		\end{split}
	\end{equation*}
	where we used the assumption $A_3 \geq A_2 > 0$ in the last inequality. This contradiction shows that $A_1 > 0$ on $M^4$.

	This completes the proof of Theorem \ref{thm:ac_A_2>0}.
\end{proof}

Next, we prove Theorem \ref{thm:ac_PIC}.

\begin{proof} [Proof of Theorem \ref{thm:ac_PIC}]
	Let $0 \leq B_1 \leq B_2 \leq B_3$ be the singular eigenvalues of the matrix $B$ and ${\lambda}_1 \leq {\lambda}_2 \leq {\lambda}_3 \leq {\lambda}_4$ be the eigenvalues of the Ricci tensor ${Rc}$.
	Then, by \cite[Lemma 2.2] {Cao-Xie:25}, the sum of the least two eigenvalues of $Rc$ is given by
	\begin{equation*}
		\lambda_1+\lambda_2 = \tfrac 1 2(R-4B_3).
	\end{equation*}
	Thus, 2-nonnegative Ricci curvature is equivalent to $u:=R-4B_3 \geq0.$
	
	Now, by Lemma \ref{lem:4Dequations} and essentially the same computations as in the proof of \cite[Theorem 3.1(a)]{Cao-Xie:25}, we have
	\begin{equation*}
		\begin{split}
			\Delta_f u &= \Delta (R-4B_3) \\
			&\leq 2\rho u- \left[ 2|Rc|^2 - 8(A_3B_3 +  C_3B_3 + 2B_1B_2) \right] \\
			&\leq 2\rho u- 8(A_2 + A_1 + C_2 + C_1)B_3.
		\end{split}
	\end{equation*} 	
	Then, by Lemma \ref{lem:4Dequations} and direct computations, for $F = f - 2\log R$, we have
	\begin{equation} \label{eq:4-1}
		\begin{split}
			\Delta_F \left( \frac{u}{R} \right) &= \frac{1}{R}\Delta_fu - \frac{u}{R^2} \Delta_f R \\
			&\leq \frac{2}{R^2} \big[u|Rc|^2 - 4R(A_2 + A_1 + C_2 + C_1)B_3 \big].
		\end{split}
	\end{equation} 
	
	Next, we prove, again by contradiction, that $(M^4,g,f)$ has 2-nonnegative Ricci curvature. Suppose instead that the Ricci curvature is not  2-nonnegative. Then, 
	\begin{equation*}
	M^- := \{ p \in M : u(p) < 0 \}
	\end{equation*}
	is nonempty, and we have
	\begin{equation} \label {eq:4-2}
		\tilde{\epsilon} := \inf_{M}\left( \frac{u}{R} \right) < 0.
	\end{equation}
	
	\smallskip
	\noindent\textbf{Case 1: Interior infimum.}  
	Suppose the negative infimum $\tilde{\epsilon}$ is attained at some point $p_0 \in M$. Then, there exists a neighborhood $\Omega \ni p_0$, such that $u<0$ on $\Omega$. By \eqref{eq:4-1} and the assumption that $(M^4,g,f)$ has PIC, so that $A_1+A_2>0$ and $C_1+C_2>0$, we have
	\begin{equation*}
		\Delta_F \left( \frac{u}{R} \right) 
		\leq \frac{2}{R^2} \left[ u|Rc|^2 - 4R(A_2 + A_1 + C_2 + C_1)B_3 \right] \leq 0
	\end{equation*}
	on $\Omega$, where we have used the fact that $B_3 \geq 0$. Thus, by Calabi's barrier strong maximum principle (Lemma \ref{lem:barrier-max}), the ratio $u/R$ must be constant on $\Omega$. In particular, we have
	\begin{equation*}
		u|Rc|^2 - 4R(A_2 + A_1 + C_2 + C_1)B_3 \equiv 0,
	\end{equation*}
	which implies $u |Rc|^2 =4R(A_2 + A_1 + C_2 + C_1)B_3 \equiv0$ on $\Omega$.  
	
	Since $u(p_0)<0$ by  \eqref{eq:4-2}, we must have $|Rc|^2=0$ at $p_0$. But this is a contradiction to the scalar curvature  $R>0$. Hence, {\bf Case 1} is ruled out.
	
	\smallskip
	\noindent\textbf{Case 2: Infimum at infinity.}  
	Suppose that $u/R$ attains its negative infimum at infinity. Since $(M^4,g,f)$ has PIC, the asymptotic cone must have WPIC. By Lemma \ref{lem:4Dcone}, the Ricci curvature of the asymptotic cone is therefore nonnegative. Then by similar arguments in the proof of Theorem \ref{thm:ac_Wflat}, we can rule out {\bf Case 2}.
	
	\smallskip
	Combining both cases, we conclude that $M^-=\emptyset$. Therefore, $M$ satisfies $u \geq 0$, completing the proof of 2-nonnegativity of the Ricci curvature.
	
	\smallskip
	Given that $(M^4,g,f)$ has 2-nonnegative Ricci curvature, following the proof in \cite[Theorem 3.1(b)]{Cao-Xie:25}, we shall prove 2-positive Ricci curvature by contradiction. We consider the quadratic form $Z:=RI-4\sqrt{B {}^tB}$, where $I$ is the 3 by 3 identity matrix. By the 2-nonnegativity of the Ricci curvature, we know that $Z\geq 0$ and that  2-positive Ricci curvature is equivalent to $Z>0$. Now, we denote the eigenvalues of $Z$ by
	\begin{equation*}
		0\leq Z_1\leq Z_2\leq Z_3.
	\end{equation*}
	Assume that $Z$ has a null eigenvector at some point $p_0$. Then $Z_1$ attains its minimum at $p_0$. Let $\eta \in \wedge_{p_0}^{+}(M)$ be a null eigenvector of $Z$ such that $Z(\eta,\eta) = Z_1(p_0) = 0$ at $p_0$. Extend $\eta$ to a local section (still denoted by $\eta$) by parallel transport along geodesics emanating from $p_0$. Then, at $p_0$, in the barrier sense, we have
	\begin{equation*}
		\begin{split}
			0 &\leq \Delta_f Z_1 \\
			&\leq \Delta_f Z(\eta, \eta) \\
			&= (\Delta_f Z)(\eta, \eta) \\
			&\leq 2\rho Z_1- 8(A_2 + A_1 + C_2 + C_1)B_3 \\
			&< 0,
		\end{split}
	\end{equation*}
	where we have used the PIC condition and $B_3(p_0)=\tfrac{R}{4}(p_0)>0$ in the last inequality. Thus, we get a contradiction. Therefore, the Ricci curvature is 2-positive.
	
	\smallskip
	Finally, we establish the curvature estimate $|{Rm}| \le 2R$. By the positive isotropic curvature condition, we have $|A|^2 \le \tfrac{3}{16}R^2$ and $|C|^2 \le \tfrac{3}{16}R^2$; see, e.g., the proof of \cite[Theorem 1.3]{Cao-Xie:23}. Since the Ricci tensor of $(M^4, g, f)$ is $2$-positive, it follows, as in the proof of \cite[Theorem 3.1(a)]{Cao-Xie:25}, that $|{Rc}|^2 \le R^2$. 
	On the other hand,
	\[
	R^2 \ge |{Rc}|^2 = |\mathring{Rc}|^2 + \tfrac{1}{4}R^2 = 4|B|^2 + \tfrac{1}{4}R^2.
	\]
	Therefore,
	\begin{equation*}\label{eq:boundedRm}
		\begin{split}
			|{Rm}|^2 
			&\le 2\bigl(|A|^2 + |C|^2 + |B|^2\bigr) 
			\le \tfrac{9}{8}R^2.
		\end{split}
	\end{equation*}
	This completes the proof of the curvature bound, and hence the proof of Theorem \ref{thm:ac_PIC}.
\end{proof}

Finally, by Lemma \ref{lem:4Dcone} and an argument similar to that used in the proof of Theorem \ref{thm:ac_PIC}, we obtain the following slightly stronger result than Theorem \ref{thm:ac_PIC}.

\begin{thm} \label{thm:ac_PIC_restate}
	Let $(M^4,g,f)$ be a $4$-dimensional non-compact asymptotically conical gradient expanding Ricci soliton. Suppose that the asymptotic cone has positive scalar curvature, and satisfies $A \geq 0$ and $C \geq 0$. Then:
	\begin{itemize}
		\item[(a)] If $(M^4,g,f)$ satisfies $A_2 \geq 0$ and $C_2 \geq 0$, then the Ricci curvature is $2$-nonnegative, and $|Rm| \leq R$.
		
		\item[(b)] If $(M^4,g,f)$ satisfies $A_2 > 0$ and $C_2 > 0$, then the Ricci curvature is $2$-positive, and $|Rm| \leq R$.
	\end{itemize}
\end{thm}

\smallskip
\section{General lemma and further applications} \label{sec:general_lem}

In this section, we formulate a fairly general method that can be effectively applied to prove generalized Hamilton--Ivey type curvature pinching estimates for a class of non-compact, asymptotically conical gradient Ricci solitons. As applications, we obtain several analogues of other known curvature pinching results for ancient solutions, in the setting of asymptotically conical expanding Ricci solitons--including Theorem \ref{thm:ac_UPIC} as stated in the introduction, and Theorem \ref{thm:curv_pinching} below.

\subsection{A general lemma} 

By identifying common patterns in how we have proved the curvature pinching theorems in Section \ref{sec:ac-halfPIC}, we are led to the following 
\begin{lem} \label{lem:general}
	Let $(M^n,g,f)$ be an $n$-dimensional non-compact gradient Ricci soliton, satisfying Eq. \eqref{eq:Riccisoliton} and with positive scalar curvature. Suppose $u: M^n \rightarrow \R$ is a Lipschitz function and satisfies the differential inequality
	\begin{equation} \label{eq:Delta_f-u}
		\Delta_f u \leq 2\rho u -v
	\end{equation}
	in the barrier sense, where  $v \geq0$ is a nonnegative function on $M^n$. 

	\begin{itemize}
		\item[(i)] If for any sequence of points $\{p_i\} \subset M^n$, with $p_i \rightarrow \infty$, we have
		\begin{equation} \label{eq:lim_u/R}
			\liminf_{i\rightarrow \infty} \left( \frac{u}{R}\right)  \geq 0.
		\end{equation}
		Then, $u\geq 0$ on $M^n$.
		\smallskip
	
		\item[(ii)] If in addition $v>0$, then $u>0$ on $M^n$.
	\end{itemize}
\end{lem}

\begin{proof}
	(i) First of all, we shall compute the differential equation of $u/R$ in the barrier sense. By direct computations, we have
	\begin{equation} \label{eq:u/R}
		\begin{split}
			\Delta_f\left(\frac{u}{R}\right) = \frac{1}{R} \Delta_f u
			- \frac{u}{R^2} \Delta_f R
			- \frac{2}{R^2}\langle \nabla u, \nabla R \rangle
			+ \frac{2u}{R^3} |\nabla R|^2.
		\end{split}
	\end{equation}
	Let $F = f - 2\log R$. Substituting \eqref{eq:Delta_f-u} and the formula for $\Delta_f R$ in Lemma \ref{lem:4Dequations} into \eqref{eq:u/R}, we obtain
	\begin{equation} \label{eq:Delta_F}
		\begin{split}
			\Delta_F \left(\frac{u}{R}\right) &= \frac{1}{R} \Delta_f u - \frac{u}{R^2} \Delta_f R \\
			&\leq \frac{1}{R} (2\rho u - v) - \frac{u}{R^2} (2\rho R - 2 |Rc|^2) \\
			&\leq \frac{1}{R^2} (2u |Rc|^2-Rv).
		\end{split}
	\end{equation}
	
	Now, we prove by contradiction again.  
	Suppose the lemma fails. Then the set
	\[
	M^- := \{ p \in M : u(p) < 0 \}
	\]
	is nonempty, and
	\begin{equation*} \label{eq:delta<0}
		\delta := \inf_{M}\left( \frac{u}{R} \right) < 0.
	\end{equation*}
	
	\noindent{\bf Case 1: Interior infimum.}  
	Suppose the negative infimum $\delta$ is attained at some $p_0 \in M$. Then, in a neighborhood $\Omega \ni p_0$, we have $u<0$ on $\Omega$. Thus, by \eqref{eq:Delta_F}, we have
	\begin{equation*}
		\begin{split}
			\Delta_F \left(\frac{u}{R}\right)
			&\leq \frac{1}{R^2} (2u |Rc|^2-Rv) \leq 0
		\end{split}
	\end{equation*}
	in the barrier sense on $\Omega$. By Calabi's barrier strong maximum principle (Lemma \ref{lem:barrier-max}), the ratio $u/R$ must be constant on $\Omega$. Since $R>0$, $v\geq 0$ and $u<0$ on $\Omega$, it follows that  
	\[
	u |Rc|^2 \equiv 0,
	\]
	which forces $|Rc| = 0$ on $\Omega$, a contradiction to $R>0$. 
	Thus, \textbf{Case 1} is ruled out.
	
	\smallskip
	\noindent{\bf Case 2: Infimum at infinity.}  
	Assume instead that $u/R$ attains its negative infimum at infinity. Then there exists a sequence $\{p_i\}\subset M$ with $p_i \to \infty$ such that
	\begin{equation*} \label{eq:lim_delta<0}
		\lim_{i\to\infty} \frac{u}{R}(p_i) = \delta < 0.
	\end{equation*}
	However, this contradicts our assumption \eqref{eq:lim_u/R}. 
	Thus, \textbf{Case 2} is impossible.  
	
	\smallskip
	Combining both cases, we conclude that $M^-=\emptyset$, so $u\geq 0$ on $M^n$.
	
	\smallskip
	(ii) We prove $u>0$ by contradiction.  Suppose $u(p_0) = 0$ at some point $p_0 \in M$. Then, since $u\geq 0$ on $M^n$, it follows that $u$ attains its minimum at $p_0$. Then, at $p_0$, in the barrier sense  
	\begin{equation*}
		\begin{split}
			0 \leq \Delta_f u \leq 2\rho u -v <0,
		\end{split}
	\end{equation*}
	where we have used the assumption $v > 0$. This is a contradiction.	
\end{proof}

\begin{rmk}
	In particular, if an expanding Ricci soliton $(M^n,g,f)$ is asymptotically conical, and if its asymptotic cone has positive scalar curvature and satisfies $u \geq 0$, then $(M^n,g,f)$ fulfills the asymptotic condition \eqref{eq:lim_u/R} in Lemma \ref{lem:general}. 
\end{rmk}

\begin{rmk}
	Similarly, if a steady Ricci soliton $(M^n,g,f)$ is asymptotic to a cylinder with $u \geq 0$, then $(M^n,g,f)$ satisfies the asymptotic condition \eqref{eq:lim_u/R} in Lemma \ref{lem:general}.
\end{rmk}

\subsection{The proof of Theorem \ref{thm:ac_UPIC}} 
In this subsection, we apply Lemma \ref{lem:general} to prove Theorem \ref{thm:ac_UPIC}.  Recall that, by \emph{uniformly PIC} we mean that $(M^4, g)$ has PIC and satisfies in addition the pointwise pinching condition  
\begin{equation*}
	\max\{A_3,B_3,C_3\} \leq \Lambda \min\{A_1+A_2, C_1+C_2\},
\end{equation*}
for some constant $\Lambda \geq 1$. 

\begin{proof} [Proof of Theorem \ref{thm:ac_UPIC}]
	First of all, it is easy to see that the inequality $B_3^2 \leq A_1C_1$ implies nonnegative curvature operator $Rm\geq 0$ for $(M^4, g)$; see, e.g., \cite[Lemma 4.4]{Cho-Li:23} for a proof. Hence, to prove Theorem \ref{thm:ac_UPIC}, it suffices to establish the inequality $B_3^2 \leq A_1C_1$. 

	We shall prove the inequality $B_3^2 \leq A_1C_1$ in three steps as in \cite{Brendle:14}. The main computations in each step below essentially come from Brendle's work \cite{Brendle:14}, in which he used the pinching estimates of Hamilton \cite{Ha:97} to show that a gradient steady Ricci soliton with UPIC must have positive curvature operator. 

	\smallskip
	\noindent{\bf Step 1.} To show $ A_3 \leq (6\Lambda^2 + 1)A_1$ and $C_3 \leq (6\Lambda^2 + 1)C_1.$
	\smallskip

	By the same computations as in \cite[Lemma 6.1]{Brendle:14}, we have
	\begin{equation*}
		\Delta_f [(6\Lambda^2+1)A_1 - A_3] \leq 2\rho [(6\Lambda^2+1)A_1 - A_3] - A_3^2.
	\end{equation*}
	
	Moreover, since the asymptotic cone $\mathcal{C}$ is a non-flat  Euclidean cone, by \eqref{eq:4Dcone_A&C} and \eqref{eq:4Dcone_B}, on the cone $\mathcal{C}$, we have
	\begin{equation*}
		\bar{A}_i=\bar{C}_j=\bar{B}_k,\quad 1\leq i,j,k \leq 3,
	\end{equation*}
	where the bar denotes the corresponding curvature quantities on the asymptotic cone $\mathcal{C}$.
	Hence, the inequality $ A_3 \leq (6\Lambda^2 + 1)A_1$ follows from Lemma \ref{lem:general} with $u=(6\Lambda^2+1)A_1 - A_3$. Similarly, we have  $C_3 \leq (6\Lambda^2 + 1)C_1.$
	
	\smallskip
	\noindent{\bf Step 2.} To show  $ 4B_3^2 \leq (A_1+A_2)(C_1+C_2).$
	
	\smallskip
	Following \cite{Brendle:14}, we prove by contradiction. Suppose that 
	\[
	\gamma = \sup_{M} \frac{2B_3}{\sqrt{(A_1+A_2)(C_1+C_2)}} > 1.
	\]
	Let $w_1:=\tfrac{1}{2}\sqrt{(A_1+A_2)(C_1+C_2)}$. By the same computations as in \cite[Lemma 6.2]{Brendle:14}, we can find a positive constant $\delta_1>0$ such that
	\begin{equation*}
		\Delta_f (\gamma w_1 - B_3-\delta_1 R) \leq 2\rho (\gamma w_1 - B_3-\delta_1 R) - \delta_1 |Rc|^2.
	\end{equation*}
	
	On the other hand,  since $\gamma>1$, on the asymptotic non-flat Euclidean cone $\mathcal{C}$ we can choose $\bar{\delta}_1>0$ small enough such that
	\begin{equation*}
		\gamma \bar{w}_1 - \bar{B}_3-\bar{\delta}_1 \bar{R} > 0.
	\end{equation*}
	Hence, by Lemma \ref{lem:general} with $u=\gamma w_1 - B_3 - \delta_1 R$, we obtain $\gamma w_1 - B_3 - \delta_1 R \geq 0$, which contradicts the definition of $\gamma$. Therefore, $\gamma \leq 1$, completing the proof of {\bf Step 2}.

	\smallskip
	\noindent{\bf Step 3.} To show $ B_3^2 \leq A_1C_1.$
	
	\smallskip
	Again, following \cite{Brendle:14}, we argue by contradiction. Suppose that 
	\[
	\gamma^{\prime} = \sup_{M} \frac{B_3}{\sqrt{A_1C_1}} > 1.
	\]
	Let $w_2:=\sqrt{A_1C_1}$. Then by the same computations as in \cite[Proposition 6.3]{Brendle:14}, we can find a positive constant $\delta_2>0$ such that
	\begin{equation*}
		\Delta_f (\gamma^{\prime} w_2 - B_3-\delta_2 R) \leq 2\rho (\gamma^{\prime} w_2 - B_3-\delta_2 R) - \delta_2 |Rc|^2.
	\end{equation*}
	
	On the other hand, since $\gamma^{\prime}>1$, on the asymptotic non-flat Euclidean cone $\mathcal{C}$ we can choose $\bar{\delta}_2>0$ small enough such that 
	\begin{equation*}
		\gamma^{\prime} \bar{w}_2 - \bar{B}_3-\bar{\delta}_2 \bar{R} > 0.
	\end{equation*}	
	Hence, as in the proof of {\bf Step 2}, {\bf Step 3} follows from Lemma \ref{lem:general} with $u=\gamma^{\prime} {w}_2 - {B}_3-{\delta}_2 {R} $.  Therefore, $(M^4,g,f)$ has nonnegative curvature operator.  

	Finally, by applying the same argument as in the proof of Theorem \ref{thm:ac_Wflat} for showing $Rm>0$, it follows that $(M^4, g, f)$ has positive curvature operator. This concludes the proof of Theorem \ref{thm:ac_UPIC}.
\end{proof}

\begin{rmk} 
	The assumption of the asymptotic cone being a (non-flat) Euclidean cone is used in the proofs of Step 2 and Step 3. Indeed, by \eqref {eq:4Dcone_A&C} and \eqref{eq:4Dcone_B}, requiring either $ 4B_3^2 \leq (A_1+A_2)(C_1+C_2)$ or $B_3^2 \leq A_1C_1$ to hold on the asymptotic cone forces the link of the cone to be a (spherical) space form.
\end{rmk}

\subsection{Additional curvature pinching results}	
Furthermore, applying Lemma \ref{lem:general}, we obtain several additional curvature pinching results for asymptotically conical expanding Ricci solitons, which are analogous to results previously established for ancient solutions by Li-Wang \cite{Li-Wang:20}, Bamler-Cabezas-Rivas-Wilking \cite{Bamler-CR-Wilking:19}, Li-Ni \cite{Li-Ni:20}, Li \cite{Li:24}, Cho-Li \cite{Cho-Li:23}, Chen \cite{ChenZN:25}, and others.

\begin{thm} \label{thm:curv_pinching}
	Let $(M^n,g,f)$ be an $n$-dimensional non-compact, asymptotically conical, gradient expanding  Ricci soliton.
	\begin{itemize}
		\item[(a)] Suppose $(M^n,g,f)$ has $2$-nonnegative curvature operator and the asymptotic cone has positive scalar curvature. Then, $(M^n,g,f)$ must have positive curvature operator.
		
		\item[(b)] Suppose $(M^n,g,f)$ has WPIC1\footnote{See, e.g., \cite[Page 97]{Bamler-CR-Wilking:19} or \cite{Brendle:10,Brendle-Schoen:09} for the definitions of WPIC1 and WPIC2.} and that the asymptotic cone has positive scalar curvature and WPIC2. Then, $(M^n,g,f)$ must have WPIC2.
		
		\item[(c)] Suppose $(M^n,g,f)$, $n\geq 5$, has WPIC and that the asymptotic cone has positive scalar curvature and $2$-nonnegative Ricci curvature. Then, $(M^n,g,f)$ must have $2$-nonnegative Ricci curvature.
		
		\item[(d)] Suppose $(M^n,g,f)$ is K\"ahler and has nonnegative orthogonal bisectional curvature\footnote{See, e.g., \cite[Page 28]{Li-Ni:20} for the definition of nonnegative orthogonal bisectional curvature.}. If the asymptotic cone has positive scalar curvature and WPIC2, then $(M^n,g,f)$ must have WPIC2.
		
		\item[(e)] Suppose $(M^n,g,f)$, $n\geq 9$, has UPIC.  If the asymptotic cone has positive scalar curvature and WPIC2, then $(M^n,g,f)$ must have WPIC2.
	\end{itemize}
\end{thm}

\begin{proof}[Sketch of Proof] 
	It suffices to verify, in each case, that the corresponding least curvature eigenvalue satisfies the differential equation \eqref{eq:Delta_f-u} in Lemma \ref{lem:general}. 
	
	\smallskip
	(a) By the same computations as in \cite[Theorem 27]{Li-Wang:20} (see also \cite[Lemma 2.4]{ChenHW:91}), the differential inequality \eqref{eq:Delta_f-u} in Lemma \ref{lem:general} for the least eigenvalue of the Riemann curvature operator is satisfied.
	
	\smallskip
	(b) By the same computations as in \cite[Lemma 4.2]{Bamler-CR-Wilking:19}, the differential inequality \eqref{eq:Delta_f-u} in Lemma \ref{lem:general} for the least eigenvalue of the complex sectional curvature is satisfied. 
	
	\smallskip
	(c) By the same computations as in \cite[Proposition 4.2]{Li-Ni:20}, the differential inequality \eqref{eq:Delta_f-u} in Lemma \ref{lem:general} for the sum of the two least eigenvalues of the Ricci tensor is satisfied.
	
	\smallskip
	(d) First of all, by the same computations as in \cite[Lemma 6.1]{Li-Ni:20}, the differential inequality \eqref{eq:Delta_f-u} in Lemma \ref{lem:general} for the least bisectional curvature is satisfied. Hence, by Lemma \ref{lem:general}, the expanding K\"ahler-Ricci soliton has nonnegative bisectional curvature. Moreover, by the same arguments as in \cite[Theorem 3.3]{Li:24}, one can show that the smallest eigenvalue of the complex sectional curvature also satisfies the differential inequality \eqref{eq:Delta_f-u} in Lemma \ref{lem:general}. 
	
	\smallskip
	(e) Finally, by the same computations as in \cite[Theorem 3.2]{Cho-Li:23} for $n \geq 12$ and in \cite{ChenZN:25} for $9 \leq n \leq 11$, the differential inequality \eqref{eq:Delta_f-u} in Lemma \ref{lem:general} for the least eigenvalue of the complex sectional curvature is satisfied. 
\end{proof}

\begin{rmk}
	Other results that have been derived using B.-L. Chen's lemma (see, e.g., \cite[Corollary 2.4]{Cho-Li:23} or \cite[Lemma 2.6]{Cao-Xie:25}) for ancient solutions can similarly be extended to the setting of asymptotically conical gradient expanding Ricci solitons, following the same approach as above.
\end{rmk}

\appendix
\section{Curvature properties of cones} \label{sec:appendix}

In this appendix, we examine the elementary relations between the curvature tensor of an $n$-dimensional cone, $n\geq 3$,  with vanishing Weyl tensor and the curvature tensor of its link. Moreover, we analyze the curvature operator decomposition of a $4$-dimensional cone. The resulting facts were used in previous sections. 

First of all, let us recall basic curvature relations between a cone and its link. For $n\geq 3$, consider any $n$-dimensional cone  
\begin{equation*}
	\mathcal{C}^n := [0,\infty) \times \Sigma^{n-1}
\end{equation*}
equipped with the Riemannian metric
\begin{equation*}
	g_{c} = dr^2 + r^2 \bar{g},
\end{equation*}
where $(\Sigma^{n-1}, \bar{g})$ is a closed $(n-1)$-dimensional Riemannian manifold.
Let $\{\bar{e}_a\}_{a\geq 2}$ be a local orthonormal frame of $T\Sigma^{n-1}$.  
We define 
\begin{equation*}
e_1 = \partial_r, \qquad e_a = r^{-1}\bar{e}_a \quad (a \geq 2),
\end{equation*}
so that $\{e_i\}_{i\geq 1}$ forms a local orthonormal frame of $T\mathcal{C}$ with respect to $g_c$. Then,  the Riemann curvature tensor\footnote{For the curvature tensor formula of a general warped product space, see O'Neill \cite{O'Neill:83}.} $Rm$ of $(\mathcal{C}^n,g_c)$ is given by

\begin{equation} \label{eq:coneRm} 
\begin{aligned} 
	R_{1ijk} &= 0, \quad 1 \leq i,j,k \leq n, \\
	R_{abcd} &= r^{-2} \left[ \bar{R}_{abcd} - \big( \bar{g}_{ac}\bar{g}_{bd} - \bar{g}_{ad}\bar{g}_{bc} \big) \right], \quad 2 \leq a,b,c,d \leq n,
\end{aligned}
\end{equation}
where $\bar{R}_{abcd}$ denotes the curvature tensor of the metric $\bar{g}$ of the link $\Sigma^{n-1}$. Moreover,  the Ricci tensor $Rc$ of $(\mathcal{C}^n,g)$ is given by
\begin{equation} \label{eq:coneRc}
\begin{aligned} 
	R_{1i} &= 0, \quad 1\leq i\leq n, \\
	R_{ab} &= r^{-2} \left[ \bar{R}_{ab} - (n-2)\bar{g}_{ab}\right] , \quad 2 \leq a,b \leq n.
\end{aligned}
\end{equation}
and the scalar curvatures of $(\mathcal{C}^n,g)$ and $(\Sigma^{n-1},\bar{g})$ are related by
\begin{equation} \label{eq:coneR}
	R = r^{-2} \left[ \bar{R} - (n-1)(n-2) \right].
\end{equation}
In addition, the nonzero Weyl curvature tensor $W$ of $(\mathcal{C}^n,g)$ is given by 

\begin{equation}\label{eq:coneW}
\begin{aligned}
	W_{1a1b} &= - \frac {1}{(n-2)r^2} \left( \bar{R}_{ab} - \frac{\bar{R}} {(n-1)}\bar{g}_{ab}\right), \quad 2 \leq a, b \leq n, \\
	W_{abcd} &= r^{-2} \ \! \overline{W}_{abcd}, \quad 2 \leq a,b,c,d \leq n. 
\end{aligned}
\end{equation}

\subsection{Curvature tensor of cones with vanishing Weyl tensor} 
First, consider 
\begin{equation*}
	\mathcal{C}^3 = (0,\infty) \times \Sigma^2, \qquad g_c = dr^2 + r^2 \bar{g},
\end{equation*}
where $(\Sigma^2,\bar{g})$ is a closed Riemannian surface.  Then, the nonzero Riemann curvature tensor components of $(\mathcal{C}^3,g_c)$ are given by
\begin{equation*}
	R_{abcd} = r^{-2} \left[ \bar{R}_{abcd} - \big( \bar{g}_{ac}\bar{g}_{bd} - \bar{g}_{ad}\bar{g}_{bc} \big) \right], 
	\qquad 2 \leq a,b,c,d \leq 3.
\end{equation*}
Since $\dim \Sigma = 2$, we have
\begin{equation*}
	\bar{R}_{abcd} = \bar{K}\big(\bar{g}_{ac}\bar{g}_{bd} - \bar{g}_{ad}\bar{g}_{bc}\big),
\end{equation*}
where $\bar{K}$ is the Gaussian curvature of $(\Sigma^2,\bar{g})$. Hence
\begin{equation*}
	R_{abcd} = r^{-2}\left(\bar{K} - 1 \right)
	\big(\bar{g}_{ac}\bar{g}_{bd} - \bar{g}_{ad}\bar{g}_{bc}\big).
\end{equation*}

Therefore, with respect to the basis 
\begin{equation*}
	\{ e_1 \wedge e_2, e_3 \wedge e_1, e_2 \wedge e_3 \}
\end{equation*}
of $\wedge^2 T \mathcal{C}$, the curvature operator is diagonal with eigenvalues
\begin{equation*} \label{eq:3Dcone}
	R_{1212}= 0, \quad R_{1313}= 0, \quad R_{2323}= r^{-2}\left(\bar{K}- 1\right).
\end{equation*}

In summary, we have the following basic fact in dimension $n=3$. 
\begin{lem} \label{lem:3Dcone}
	Let $ \mathcal{C}^3 := [0,\infty) \times \Sigma^2 $ be a $3$-dimensional cone equipped with the Riemannian metric $g_{c} = dr^2 + r^2 \bar{g}_{\Sigma}$. Then, the only possible nonzero eigenvalue of the curvature operator $Rm$ is  the principal sectional curvature 
	\[m:= r^{-2}\left(\bar{K}- 1\right),\]
	where $\bar{K}$ is the Gaussian curvature of $(\Sigma^2,\bar{g})$. In particular, $(\mathcal{C}^3, g_{c} )$ has nonnegative curvature operator $Rm\geq 0$ if and only if it has nonnegative scalar curvature $R\geq 0$, or equivalently, if and only if $\bar{K}\geq 1$. 
\end{lem}

Meanwhile, for $n\geq 4$, consider any $n$-dimensional cone  
\begin{equation*}
	\mathcal{C}^n := [0,\infty) \times \Sigma^{n-1}
\end{equation*}
equipped with the Riemannian metric
\begin{equation*}
	g_{c} = dr^2 + r^2 \bar{g},
\end{equation*}
where $(\Sigma^{n-1}, \bar{g})$ is a closed $(n-1)$-dimensional Riemannian manifold. Then, by \eqref{eq:coneW}, we see that $\mathcal{C}$ has vanishing Weyl curvature $W=0$ if and only if its link $\Sigma$ is a space form. Thus,  we immediately have the following

\begin{lem} \label{lem:W=0cone}
	Let $\mathcal{C}^n := [0,\infty) \times \Sigma^{n-1} $ be an $n$-dimensional ($n\geq 4$) cone with nonnegative scalar curvature, equipped with the Riemannian metric $g_{c} = dr^2 + r^2 \bar{g}_{\Sigma}$. Then, $(\mC^n, g_c)$ is locally conformally flat but non-flat if and only if the link $(\Sigma^{n-1}, \bar{g})$ is a spherical space form, i.e., up to scaling, $(\Sigma^{n-1},\bar{g})$ is isometric to a quotient of the round sphere $\rS^{n-1}$. In particular, if $(\mC^n, g_c)$ is locally conformally flat and has nonnegative scalar curvature, then it has nonnegative curvature operator $Rm\geq 0$. 
\end{lem}

\subsection{Curvature decomposition and curvature operator of 4D cones} 
Recall that, with respect to the decomposition
\begin{equation*} \label{eq:decompof2forms}
	\wedge^2 = \wedge^{+}\oplus \wedge^{-}
\end{equation*} 
on any oriented smooth Riemannian 4-manifold $(M^4, g)$, the curvature operator of $(M^4, g)$ admits the following decomposition:
\begin{equation} \label{eq:curv_op_decomp}
	Rm = 
	\begin{pmatrix}
		A & B \\
		B^t & C
	\end{pmatrix}
	=
	\begin{pmatrix}
		W^+ + \frac{R}{12}I& \mathring{Rc} \\
		\mathring{Rc} & W^- + \frac{R}{12}I
	\end{pmatrix},
\end{equation}
where $W^{\pm}$ denote the self-dual and anti-self-dual Weyl curvature tensors, respectively, and $\mathring{Rc}$ denotes the traceless Ricci tensor.

We may choose a basis for $\wedge_p^+$ and for $\wedge_p^-$ as follows:
\begin{equation*}
	\begin{split}
		\vphi^+_1 = \tfrac{1}{\sqrt{2}}(e_1\wedge e_2 + e_3\wedge e_4),  \\
		\vphi^+_2 = \tfrac{1}{\sqrt{2}}(e_1\wedge e_3 + e_4\wedge e_2),  \\
		\vphi^+_3 = \tfrac{1}{\sqrt{2}}(e_1\wedge e_4 + e_2\wedge e_3),  \\
	\end{split}
	\quad \quad
	\begin{split}
		\vphi^-_1 = \tfrac{1}{\sqrt{2}}(e_1\wedge e_2 - e_3\wedge e_4),  \\
		\vphi^-_2 = \tfrac{1}{\sqrt{2}}(e_1\wedge e_3 - e_4\wedge e_2),  \\
		\vphi^-_3 = \tfrac{1}{\sqrt{2}}(e_1\wedge e_4 - e_2\wedge e_3),  \\
	\end{split}
\end{equation*}
where $\{e_1,e_2,e_3,e_4\}$ is any positively oriented orthonormal tangent frame at a point $p$. Here, we have used the metric $g$ to identify the tangent space and the cotangent space at $p$. The inner product on 2-forms is defined by
\begin{equation} \label{eq:metric_2forms}
	\la X\wedge Y, V\wedge W \ra = \la X,V\ra \la Y,W\ra - \la X,W\ra \la Y,V\ra .
\end{equation}

Observe that, for the matrices $A$ and $C$ in \eqref{eq:curv_op_decomp}, we have
\begin{equation} \label{eq:matrixA&C}
	\begin{split}
		A_{11}=\tfrac{1}{2}\left( R_{1212}+R_{3434}+2R_{1234} \right), \\
		A_{22}=\tfrac{1}{2}\left( R_{1313}+R_{4242}+2R_{1342} \right), \\
		A_{33}=\tfrac{1}{2}\left( R_{1414}+R_{2323}+2R_{1423} \right), \\
	\end{split}
	\quad \quad
	\begin{split}
		C_{11}=\tfrac{1}{2}\left( R_{1212}+R_{3434}-2R_{1234} \right), \\
		C_{22}=\tfrac{1}{2}\left( R_{1313}+R_{4242}-2R_{1342} \right), \\
		C_{33}=\tfrac{1}{2}\left( R_{1414}+R_{2323}-2R_{1423} \right). \\
	\end{split}
\end{equation}
For the matrix $B$, we have
\begin{equation} \label{eq:matrixB}
	\begin{split}
		B_{11}=\tfrac{1}{2}\left( R_{1212}-R_{3434} \right),  \\
		B_{22}=\tfrac{1}{2}\left( R_{1313}-R_{4242} \right),  \\
		B_{33}=\tfrac{1}{2}\left( R_{1414}-R_{2323} \right), \\
	\end{split}
	\quad or \quad
	\begin{split}
		B_{11}=\tfrac{1}{4}\left( R_{11}+R_{22}-R_{33}-R_{44} \right), \\
		B_{22}=\tfrac{1}{4}\left( R_{11}+R_{33}-R_{44}-R_{22} \right), \\
		B_{33}=\tfrac{1}{4}\left( R_{11}+R_{44}-R_{22}-R_{33} \right), \\
	\end{split}
\end{equation}
and
\begin{equation*}
	B_{12}=\tfrac{1}{2}\left(R_{1213}+R_{3413}-R_{1242}-R_{3442}\right) =\tfrac{1}{2}\left(R_{23}-R_{14}\right),\ etc.
\end{equation*}

Now, we consider any $4$-dimensional cone
\begin{equation*}
	\mathcal{C}^4 := [0,\infty) \times \Sigma^3
\end{equation*}
equipped with the Riemannian metric
\begin{equation*}
	g_{c} = dr^2 + r^2 \bar{g},
\end{equation*}
where $(\Sigma^3, \bar{g})$ is a closed $3$-dimensional Riemannian manifold. On $(\Sigma^3,\bar{g})$, diagonalize the curvature operator $\overline{Rm}$ with respect to the local 2-frame 
\begin{equation*}
	\{ \bar{e}_2 \wedge \bar{e}_3, \bar{e}_3 \wedge \bar{e}_4, \bar{e}_4 \wedge \bar{e}_2 \}
\end{equation*}
of $\wedge^2 T\Sigma^3$, where $\{\bar{e}_2,\bar{e}_3,\bar{e}_4\}$ is a local orthonormal frame of $T\Sigma^3$.  
Suppose that, in this frame, $\overline{Rm}$ is diagonal with entries
\begin{equation*}
	\bar{R}_{2323} =: m_3, \quad 
	\bar{R}_{2424} =: m_2, \quad 
	\bar{R}_{3434} =: m_1
\end{equation*}
such that $m_1\leq m_2\leq m_3$. Then, with respect to the tangent frame $\{\bar{e}_2,\bar{e}_3,\bar{e}_4\}$, the Ricci tensor $\overline{Rc}$ of $\Sigma^3$ takes the form
\begin{equation*} \label{eq:Ricci3D}
	\overline{Rc} =
	\begin{pmatrix}
		m_1 + m_2 & 0 & 0 \\
		0 & m_1 + m_3 & 0 \\
		0 & 0 & m_2 + m_3
	\end{pmatrix}.
\end{equation*}
The scalar curvature of $\Sigma^3$ is given by 
\begin{equation*} \label{eq:Scalar3D}
	\bar{R} = 2(m_1 + m_2 + m_3).
\end{equation*}

For the cone $\mathcal{C}^4$, we choose 
\begin{equation*}
	e_1 = \partial_r, \qquad e_i = r^{-1}\bar{e}_i \ (i=2,3,4),
\end{equation*}
so that $\{e_1,e_2,e_3,e_4\}$ forms an orthonormal frame of $T\mathcal{C}$ with respect to $g_c$. Then, by \eqref{eq:coneRm}, we have 
\begin{align*}
	R_{1jkl} &= 0, \quad 1 \leq j,k,l \leq 4, \\
	R_{abcd} &= r^{-2} \left[ \bar{R}_{abcd} - \big( \bar{g}_{ac}\bar{g}_{bd} - \bar{g}_{ad}\bar{g}_{bc} \big) \right], \quad 2 \leq a,b,c,d \leq 4
\end{align*}
where $\bar{R}_{abcd}$ denotes the curvature tensor of the link $(\Sigma^{3},\bar{g})$.

Therefore, by \eqref{eq:matrixA&C}, we obtain
\begin{equation} \label{eq:4Dcone_A&C}
	\begin{aligned}
		A_{11} = C_{11} &= \tfrac{1}{2}R_{3434} = \tfrac{1}{2r^2}(m_1-1), \\
		A_{22} = C_{22} &= \tfrac{1}{2}R_{2424} = \tfrac{1}{2r^2}(m_2-1), \\
		A_{33} = C_{33} &= \tfrac{1}{2}R_{2323} = \tfrac{1}{2r^2}(m_3-1),
	\end{aligned}
\end{equation}
and $A_{ij}=C_{ij}=0  \ (i\neq j)$, e.g., 
\begin{equation*}
	A_{12} = C_{12}= \tfrac{1}{2}R_{3442} = 0, \quad \text{etc.}
\end{equation*}

Similarly, by \eqref{eq:matrixB}, we have
\begin{equation} \label{eq:4Dcone_B}
	\begin{aligned}
		B_{11} &= -\tfrac{1}{2}R_{3434} = -\tfrac{1}{2r^2}(m_1-1)=-A_{11}, \\
		B_{22} &= -\tfrac{1}{2}R_{4242} = -\tfrac{1}{2r^2}(m_2-1)=-A_{22}, \\
		B_{33} &= -\tfrac{1}{2}R_{2323} = -\tfrac{1}{2r^2}(m_3-1)=-A_{33},
	\end{aligned}
\end{equation}
and $B_{ij}=0  \ (i\neq j)$, e.g., 
\begin{equation*}
	B_{12} = -\tfrac{1}{2}R_{3442} = 0, \quad \text{etc.}
\end{equation*}

Moreover, for the Ricci tensor of $(\mathcal{C},g_c)$, we have  $R_{ij}=0 \ (i\neq j)$ and 
\begin{align*}
		R_{11}&=0, &R_{22}=r^{-2} \left( m_2 + m_3 - 2 \right), \\
		R_{33}&=r^{-2} \left( m_1 + m_3 - 2 \right),
		&R_{44}=r^{-2} \left( m_1 + m_2 - 2 \right).
\end{align*}

\smallskip
In conclusion, based on the above computations, we have
\begin{lem} \label{lem:4Dcone}
	Let $ \mathcal{C}^4 := [0,\infty) \times \Sigma^3 $ be a $4$-dimensional cone equipped with the Riemannian metric $g_{c} = dr^2 + r^2 \bar{g}_{\Sigma}$.
	Then, the following statements hold:
	\begin{itemize}
		\item[(i)] $(\mathcal{C}^4, g_c)$ has $A \geq 0$ if and only it has nonnegative curvature operator $Rm\geq 0$, if and only if $\overline{Rm} \geq Id$. 
	
		\smallskip	
		\item[(ii)] $(\mathcal{C}^4, g_c)$ has half-WPIC  if and only if it has WPIC and nonnegative Ricci curvature $Rc\geq 0$, if and only if $\overline{Rc} \geq 2 {\bar{g}}$.
	\end{itemize}
\end{lem}

\bigskip

\end{document}